\documentclass[reqno, 12pt,a4letter]{amsart}
\usepackage{amsmath,amsxtra,amssymb,latexsym, amscd,amsthm}
\usepackage{graphicx,color}
\usepackage{subfigure}
\usepackage[utf8]{inputenc}
\usepackage{enumerate}
\usepackage[mathscr]{euscript}
\usepackage{mathrsfs}
\usepackage[english]{babel}
\usepackage[euler]{textgreek}
\usepackage{color}
\usepackage{cite}
\newtheorem {theorem}{Theorem}[section]
\newtheorem {corollary}[theorem]{Corollary}
\newtheorem {proposition}[theorem]{Proposition}
\newtheorem {lemma}[theorem]{Lemma}

\newtheorem {example}[theorem]{Example}
\newtheorem {definition}[theorem]{Definition}
\newtheorem {remark}[theorem]{Remark}

\newcommand {\Limsup} {\mathop{{\rm Lim\,sup}\,}}

\def\ar{a\kern-.370em\raise.16ex\hbox{\char95\kern-0.53ex\char'47}\kern.05em}
\def\ees{{\accent"5E e}\kern-.385em\raise.2ex\hbox{\char'23}\kern-.08em}
\def\eex{{\accent"5E e}\kern-.470em\raise.3ex\hbox{\char'176}}
\def\AR{A\kern-.46em\raise.80ex\hbox{\char95\kern-0.53ex\char'47}\kern.13em}
\def\EES{{\accent"5E E}\kern-.5em\raise.8ex\hbox{\char'23 }}
\def\EEX{{\accent"5E E}\kern-.60em\raise.9ex\hbox{\char'176}\kern.1em}
\def\ow{o\kern-.42em\raise.82ex\hbox{
		\vrule width .12em height .0ex depth .075ex \kern-0.16em \char'56}\kern-.07em}
\def\OW{O\kern-.460em\raise1.36ex\hbox{
		\vrule width .13em height .0ex depth .075ex \kern-0.16em \char'56}\kern-.07em}
\def\UW{U\kern-.42em\raise1.36ex\hbox{
		\vrule width .13em height .0ex depth .075ex \kern-0.16em \char'56}\kern-.07em}
\def\B {\mathbb{B}}

\title{Subdifferentials at infinity and applications in optimization}

\author{DO SANG KIM} 
\address{Department of Applied Mathematics, Pukyong National University, Korea}
\email{dskim@pknu.ac.kr}

\author{MINH T\`UNG NGUY\EEX N} 
\address{Faculty of Mathematical Economics, Banking University of Ho Chi Minh City, Ho Chi Minh City,Vietnam}
\email{tungnm@hub.edu.vn}

\author{TI\EES N-S\OW N PH\d{A}M}
\address{Department of Mathematics, Dalat University, 1 Phu Dong Thien Vuong, Dalat, Vietnam}
\email{sonpt@dlu.edu.vn}

\date{ \today}
\thanks{}

\begin{document}
	\maketitle
	
	\begin{abstract} 
		In this work, the notions of {\em normal cones at infinity} to unbounded sets and {\em limiting and singular subdifferentials at infinity} for extended real value functions are introduced. Various calculus rules for these notions objects are established. A complete characterization of the Lipschitz continuity at infinity for lower semi-continuous functions is given. The obtained results are aimed ultimately at applications to diverse problems of optimization, such as optimality conditions, coercive properties, weak sharp minima and stability results.
	\end{abstract}
	
	\section{Introduction}
	
	The theory of subdifferentials (or generalized derivatives) is concerned with the differential properties of functions, which do not have derivatives in the usual sense. This theory, in the setting of finite dimensional spaces, associates with a function $f \colon \mathbb{R}^n \to \overline{\mathbb{R}} := \mathbb{R} \cup \{\infty\}$ and a point $x \in \mathbb{R}^n$, a certain (not necessarily convex) subset $\partial f(x)$ of $\mathbb{R}^n,$ 
	which is called the {\em subdifferential} of $f$ at $x$ and its elements are called {\em subgradients} of $f$ at $x.$ The theory found widespread application in optimization and problems of control and partial differential equations, as seen for instance in the books \cite{Clarke1990, Clarke1998, Ioffe2017, Mordukhovich2006, Mordukhovich2018, Penot2013, Rockafellar1998} with the references therein.

We will not discuss in detail aspects of the theory of subdifferentials and just recall that, 
% the subdifferential of convex functions was defined by Rockafellar \cite{Rockafellar1963}, Clarke's subdifferential was proposed by Clarke \cite{Clarke1973}, the limiting subdifferential was introduced by Mordukhovich \cite{Mordukhovich1976-1}. We also note that, 
under certain assumptions, rules such as the necessary optimality condition $0 \in \partial f(x)$ are valid.
	
Very recently, following the ideas of Clarke \cite{Clarke1973} and Rockafellar \cite{Rockafellar1963}, the second and third authors of this paper introduced in \cite{PHAMTS2023-4} several novel notions of variational analysis, including Clarke's tangent and normal cones, subgradients, Lipschitz continuity, etc., {\em at infinity;} they studied the fundamental properties of these notions and derived necessary optimality conditions for optimization problems in scenarios where objective functions are bounded from below but do not have a global minimum, in which existing subdifferentials Fermat rules do not apply.
	
	The present paper concerns to the behavior of  sets and functions at infinity; however, our approach here is different from the one in \cite{PHAMTS2023-4}. Indeed, following the idea of Mordukhovich \cite{Mordukhovich1976-1}, we define and study {\em normal cones at infinity} to unbounded sets as well as {\em limiting and singular subdifferentials at infinity} for extended real value functions. The obtained results are aimed ultimately at applications to diverse problems of optimization following the now-familiar pattern for subdifferentials in the classical case.
	
	\subsection*{Contributions}
	Given an unbounded closed set $\Omega \subset \mathbb{R}^n$ and a lower semi-continuous function $f \colon \mathbb{R}^n \to \overline{\mathbb{R}},$ our main contributions are as follows:
	\begin{itemize}
		\item We define and study the {\em normal cone to $\Omega$ at infinity}, denoted by $N_\Omega(\infty),$ and the {\em limiting and singular subdifferentials of $f$ at infinity}, denoted by $\partial f(\infty)$ and $\partial^\infty f(\infty).$ We also provide a number of rules for computing of these notions.
		
		\item We show that the function $f$ is {\em Lipschitz at infinity} if and only if its singular subdifferential at infinity reduces to $\{0\},$ and under this equivalence, the {\em Clarke subdifferential of $f$ at infinity} introduced in \cite{PHAMTS2023-4} is equal to the convex hull of the limiting subdifferential of $f$ at infinity.
		
		\item Assume that $f$ is bounded from below on $\Omega.$ If the optimization problem
		\begin{equation*}
		\mathrm{minimize } \ f (x) \quad \textrm{ over } \quad x \in \Omega. \tag{P}
		\end{equation*}
		has no solution, then the {\em necessary optimality condition at infinity} holds: 
		$$0 \in \partial f(\infty) + N_{\Omega}(\infty).$$ 
		
		\item Assume that $0 \not \in \partial f(\infty)+ N_{\Omega}(\infty).$ We have the following statements: 
		
		\begin{itemize}
			\item The function $f$ is coercive on $\Omega,$ the solution set $\mathrm{Sol}$ of the problem~(P) is nonempty compact and there exist constants $c > 0$ and $R > 0$ satisfying the {\em weak sharp minima at infinity}:
			\begin{eqnarray*}
				f(x) - \min_{x \in \Omega} f(x) &\ge& c\, \mathrm{dist} (x, \mathrm{Sol}) \quad \textrm{ for all } \quad x \in \Omega \ \textrm{ with } \ \|x\| > R,
			\end{eqnarray*}
			where $\mathrm{dist} (x, \mathrm{Sol})$ stands for the distance from $x$ to $\mathrm{Sol}.$ 
			\item  For all $u \in \mathbb{R}^n$ near $0,$ the solution set $\mathrm{Sol}(u)$ in the perturbed optimization problem
			\begin{equation*}
			\mathrm{minimize } \ f_u(x) := f(x) - \langle u, x \rangle  \quad \textrm{ over } \quad x \in \Omega, \tag{P$_u$}
			\end{equation*}
			is nonempty compact and satisfies the inclusion
			\begin{eqnarray*}
				\Limsup_{u \to 0} \mathrm{Sol}(u) &\subset& \mathrm{Sol}(0).
			\end{eqnarray*}
		\end{itemize}
\end{itemize}
	
It should be noted that, due to the unboundedness and nonconvexity of neighborhoods at infinity, smooth functions are not necessarily Lipschitz at infinity, the limiting subdifferential at infinity of a smooth function may not be a singleton set, and that the limiting subdifferential at infinity of a convex function is not necessarily a convex set; for more details, see Example~\ref{Example4.7} below.
	
Also note that some results given in this paper can be extended to infinite dimensional spaces. However, to lighten the exposition, we do not pursue this extension here.
	
The rest of this paper is organized as follows. Some definitions and preliminary results from variational analysis are recalled in Section~\ref{Section2}.  Normal cones at infinity to unbounded subsets of $\mathbb{R}^n$ and subdifferentials at infinity for extended real-valued functions on $\mathbb{R}^n$ are introduced and studied in Sections~\ref{Section3} and \ref{Section4}, respectively. The Lipschitz continuity at infinity for lower semi-continuous functions is investigated in Section~\ref{Section5}. Finally, necessary optimality conditions and weak sharp minima at infinity as well as stability results for optimization problems are derived in Section~\ref{Section6}.

\section{Preliminaries} \label{Section2}

\subsection{Notation} 
Throughout this work we deal with the Euclidean space $\mathbb{R}^n$ equipped with the usual scalar product $\langle \cdot, \cdot \rangle$ and the corresponding norm $\| \cdot\|.$ We denote by $\mathbb{B}_r(x)$ the closed ball centered at $x$ with radius $r;$  when ${x}$ is the origin of $\mathbb{R}^n$ we write $\mathbb{B}_{r}$ instead of $\mathbb{B}_{r}({x}),$ and when $r = 1$ we write  $\mathbb{B}$ instead of $\mathbb{B}_{1}.$ We will adopt the convention that $\inf \emptyset = +\infty$ and $\sup \emptyset = -\infty.$

For a nonempty set $\Omega \subset \mathbb{R}^n,$ the closure, convex hull, conic hull, and positive hull of $\Omega$ are denoted, respectively, by 
$\mathrm{cl}\, {\Omega},$ $ \mathrm{co}\, \Omega,$ $\mathrm{cone}\, \Omega,$ and $\mathrm{pos}\, \Omega.$ We say that $\Omega$ is {\em locally closed} if for each $x \in \Omega$ there exists a neighborhood $U$ of ${x}$ such that $\Omega \cap U$ is a closed set. 

Let $\mathbb{R}_+ := [0, +\infty)$ and $\overline{\mathbb{R}} := \mathbb{R} \cup \{+\infty\}.$ For an extended real-valued function $f \colon \mathbb{R}^n \rightarrow \overline{\mathbb{R}},$ we denote its {\em effective domain} and {\em epigraph} by,  respectively,
\begin{eqnarray*}
\mathrm{dom} f &:=& \{ x \in \mathbb{R}^n \ | \ f(x) < \infty  \},\\
\mathrm{epi} f &:=& \{ (x, r) \in \mathbb{R}^n \times \mathbb{R} \ | \ f(x) \le r \}.
\end{eqnarray*}
We call $f$ a {\em proper} function if $f(x) < \infty$ for at least one $x \in \mathbb{R}^n,$ or in other words, if $\mathrm{dom} f$ is a nonempty set.
The function $f$ is said to be {\em lower semi-continuous} if for each $x \in \mathbb{R}^n$ the inequality $\liminf_{x' \to {x}} f(x') \ge f({x})$ holds.

Given a nonempty set $\Omega \subset \mathbb{R}^n,$ associate with it the {\em distance function}
$$\mathrm{dist}(x, \Omega) = d_{\Omega}(x) := \inf_{y \in \Omega} \|x - y\|, \quad x \in \mathbb{R}^n,$$ 
and define the {\em Euclidean projector of} $x \in \mathbb{R}^n$ to $\Omega$ by
\begin{eqnarray*}
\Pi_{\Omega}(x) &:=& \{y \in \Omega \ | \  \|x - y \| = d_{\Omega}(x)\}.
\end{eqnarray*}
The {\em indicator function} $\delta_{\Omega} \colon \mathbb{R}^n \to \overline{\mathbb{R}}$ of the set $\Omega \subset \mathbb{R}^n$ is defined by
\begin{eqnarray*}
\delta_\Omega(x) &:=&
\begin{cases}
0 & \textrm{ if } x \in \Omega, \\
\infty & \textrm{ otherwise.}
\end{cases}
\end{eqnarray*}
By definition, $\Omega$ is closed if and only if $\delta_\Omega$ is lower semi-continuous.

The {\em Painlev\'e--Kuratowski outer limit} of a set-valued map $\Phi \colon \mathbb{R}^n \rightrightarrows \mathbb{R}^m$ is defined by
\begin{eqnarray*}
\Limsup_{x' \to {x}} \Phi(x') &:=& \{y \in \mathbb{R}^m \mid \exists x_k \to {x}, \exists y_k \in \Phi(x_k), y_k \to y\}.
\end{eqnarray*}

\subsection{Normal cones and subdifferentials}

Here we recall some definitions and properties of normal cones to sets and subdifferentials of real-valued functions, which can be found in~\cite{Mordukhovich2006, Mordukhovich2018, Rockafellar1998}.

\begin{definition}{\rm Consider a set $\Omega\subset\mathbb{R}^n$ and a point ${x} \in \Omega.$
\begin{enumerate}
\item[(i)]  The {\em regular normal cone} (known also as the {\em prenormal} or {\em Fr\'echet normal cone}) $\widehat{N}_{\Omega}({x})$ to $\Omega$ at ${x}$ consists of all vectors $\xi \in\mathbb{R}^n$ satisfying
\begin{eqnarray*}
\langle \xi , x' - {x} \rangle &\le& o(\|x' -  {x}\|) \quad \textrm{ as } \quad x' \to {x} \quad \textrm{ with } \quad x' \in \Omega.
\end{eqnarray*}

\item[(ii)] 
The {\em limiting normal cone} (known also as the {\em basic} or {\em Mordukhovich normal cone}) $N_{\Omega} ({x})$ to $\Omega$ at ${x}$ consists of all vectors $\xi  \in \mathbb{R}^n$ such that there are sequences $x_k \to {x}$ with $x_k \in \Omega$ and $\xi _k \rightarrow \xi $ with $\xi _k \in \widehat N_{\Omega}(x_k),$ or in other words, 
\begin{eqnarray*}
{N}_\Omega({x}) & := &  \Limsup_{x' \xrightarrow{\Omega} {x}}\widehat{N}_\Omega({x}'),
\end{eqnarray*}
where $x' \xrightarrow{\Omega} {x}$ means that $x' \rightarrow {x} $ with $x' \in \Omega.$ 
\end{enumerate}

If $x \not \in \Omega,$ we put $\widehat{N}_{\Omega}({x}) := \emptyset$ and ${N}_\Omega({x}) := \emptyset.$
}\end{definition}

\begin{remark}{\rm
(i) It is well-known that $N_\Omega(\overline{x})$ is closed cone (may be non-convex) while $\widehat{N}_\Omega(\overline{x})$ is closed convex cone.

(ii) If $\Omega$ is a manifold of class $C^1,$ then for every point $x \in \Omega,$ the normal cones $\widehat{N}({x}; \Omega)$ and $N({x}; \Omega)$ are equal to the normal space to $\Omega$ at ${x}$ in the sense of differential geometry; see \cite[Example~6.8]{Rockafellar1998}. 

}\end{remark}

\begin{lemma} \label{Lemma2.3}
Let $\Omega \subset \mathbb{R}^n$ be a locally closed set. Then for any $x \in \Omega,$ we have the following relationships
\begin{eqnarray*}
\widehat{N}_{\Omega}(x) &=& \{ \xi \in \mathbb{R}^n \mid \langle \xi, v \rangle \le 0 \quad \textrm{ for all } \quad v \in T_{\Omega}(x) \},\\
N_{\Omega}(x) &=& \Limsup_{x' \rightarrow x} \Big[\mathrm{cone} \big(x' - \Pi_{\Omega}(x') \big) \Big],
\end{eqnarray*}
where $T_{\Omega}(x)$ stands for the {\em contingent cone} to $\Omega$ at $x \in \Omega,$ i.e.,
\begin{eqnarray*}
T_{\Omega}(x) &:=& \Limsup_{t \searrow 0} \frac{\Omega - x}{t}.
\end{eqnarray*}
\end{lemma}

Functional counterparts of normal cones are subdifferentials.
\begin{definition}{\rm
Consider a function $f\colon\mathbb{R}^n \to \overline{\mathbb{R}}$ and a point ${x} \in \mathrm{dom} f.$
\begin{enumerate}[{\rm (i)}]
\item The {\em regular} (or {\em Fr\'echet}) {\em subdifferential} of $f$ at ${x}$ is 
$$\widehat{\partial}f({x}) :=\{ u \in \mathbb{R}^n \mid (u,-1)\in \widehat{N}_{\mathrm{epi} f}({x},f({x}))  \}.  $$
\item The {\em limiting} (or {\em Mordukhovich}) {\em subdifferential} of $f$ at ${x}$ is 
$$\partial f({x}):=\{ u \in \mathbb{R}^n \mid (u,-1)\in {N}_{\mathrm{epi} f}({x},f({x}))  \}.$$
\item The {\em singular subdifferential} of $f$ at ${x}$ is 
$$\partial^\infty f({x}) := \{ u \in \mathbb{R}^n \mid (u,0)\in {N}_{\mathrm{epi} f}({x},f({x}))  \}.$$
\end{enumerate}
If $x \not \in \mathrm{dom} f,$ we put $\widehat{\partial}f({x}) := \emptyset,$ $\partial f({x}) := \emptyset$, and $\partial^\infty f({x}) := \emptyset.$ 
}\end{definition}

\begin{remark}{\rm
In \cite{Mordukhovich2006, Mordukhovich2018, Rockafellar1998} the reader can find equivalent analytic descriptions of the limiting subdifferential $\partial f({x})$ and comprehensive studies of it and related constructions. For convex $f,$ this subdifferential coincides with the convex subdifferential. Furthermore, if the function $f$ is of class $C^1,$ then $\partial f({x}) = \{\nabla f({x})\}.$ The singular subdifferential $\partial^\infty f({x})$ plays an entirely different role--it detects horizontal ``normal'' to the epigraph--and it plays a decisive role in subdifferential calculus.
}\end{remark}

The following lemmas are well known.

\begin{lemma} \label{Lemma2.6}
For a lower semi-continuous function $f \colon \mathbb{R}^n \to \overline{\mathbb{R}}$ and a point $x$ where $f$ is finite, the sets 
$\widehat{\partial}  f(x)$ and ${\partial} f(x)$ are closed, with $\widehat{\partial} f(x)$ convex and $\widehat{\partial} f(x) \subset \partial f(x).$ Furthermore, $\partial^\infty f(x)$ is closed cone and it holds that
\begin{eqnarray*}
\partial f({x}) &=& \Limsup_{x' \xrightarrow{f} {x}}\widehat{\partial} f(x'),
\end{eqnarray*}
where $x' \xrightarrow{f} {x}$ means that $x' \to {x}$ and $f(x') \to f({x}).$
\end{lemma}

\begin{lemma} \label{Lemma2.7}
For any set $\Omega \subset \mathbb{R}^n$ and point ${x} \in \Omega,$ we have
\begin{eqnarray*}
\partial \delta_\Omega({x}) &=&  \partial^\infty \delta_\Omega({x})  \ = \ N_{\Omega}({x}).
\end{eqnarray*}
\end{lemma}
\begin{proof}
See \cite[Proposition~1.19]{Mordukhovich2018}.
\end{proof}

\begin{lemma}[Fermat rule] \label{Lemma2.8}
If a proper function $f \colon \mathbb{R}^n \to \overline{\mathbb{R}}$ has a local minimum at $\overline{x},$ then $0 \in \partial f(\overline{x}).$	
\end{lemma}

\begin{lemma}\label{Lemma2.9} 
Let $f_i \colon \mathbb{R}^n \to \overline{\mathbb{R}}$, $i=1,\dots,m$ with $m \geq 2$, be lower semi-continuous at $\overline{x} \in \mathbb{R}^n$  and let all but one of these functions be Lipschitz around $\overline{x}.$ Then the following inclusions hold:
\begin{eqnarray*}
\partial \left( f_1+\cdots+f_m\right)(\overline{x}) & \subset & \partial f_1(\overline{x})+\cdots+\partial f_m(\overline{x}), \\
\partial (\max f_i) (\overline{x}) & \subset & \left\{    \sum_{i \in I(\overline{x})} \lambda_i \partial f_i(\overline{x}) \mid \lambda_i \ge 0, \sum_{i \in I(\overline{x})} \lambda_i=1  \right\}, 
\end{eqnarray*}
where $I(\overline{x}):=\{ i \in \{1,\ldots,m\}\mid f_i(\overline{x})= \max_j f_j (\overline{x})\}.$
\end{lemma}

For a function $g \colon \mathbb{R}^n \to \mathbb{R}^m$ and a vector $u \in \mathbb{R}^m,$ we define the function
$\langle u, g \rangle \colon \mathbb{R}^n \to \mathbb{R}$ by $\langle u, g \rangle (x) := \langle u, g(x) \rangle$ for $x \in \mathbb{R}^n.$ 

\begin{lemma}\label{Lemma2.10}
Let $f \colon \mathbb{R}^m \to \overline{\mathbb{R}}$ be a lower semi-continuous function and $g \colon \mathbb{R}^n \to \mathbb{R}^m$ be  locally Lipschitz around $\overline{x} \in \mathbb{R}^n.$ For $\overline{y} := g(\overline{x}),$ impose the qualification condition
$$\Big[u \in \partial^{\infty}f (\overline{y}) \quad \textrm{ and } \quad 0 \in \partial \langle u, g \rangle (\overline{x}) \Big] \quad \implies \quad u = 0.$$
Then the following subdifferential chain rules hold
$$\partial (f \circ g)(\overline{x}) \subseteq \bigcup_{u \in \partial f(\overline{y})} \partial \langle u, g \rangle (\overline{x}), \quad \partial^{\infty} (f \circ g)(\overline{x}) \subseteq \bigcup_{u \in \partial^{\infty} f(\overline{y})} \partial \langle u, g \rangle (\overline{x}). $$
\end{lemma}

Finally, we recall the Ekeland variational principle (see \cite{Ekeland1974, Ekeland1979}).

\begin{lemma}\label{Lemma2.11} 
Let $f \colon \mathbb{R}^n \to \overline{\mathbb{R}}$ be a proper, lower semi-continuous and bounded from below function. Let $\epsilon >0$ and $x_0 \in \mathbb{R}^n$ be given such that 
\begin{eqnarray*}
f (x_0) &\le& \inf_{x \in \mathbb{R}^n}f (x) + \epsilon. 
\end{eqnarray*}
Then, for any $\lambda >0$ there is a point $x_1 \in \mathbb{R}^n$ satisfying the following conditions
\begin{enumerate}[{\rm (i)}]
\item $f (x_1) \le f(x_0),$
\item $\|x_1-x_0\| \le \lambda,$ and
\item $f (x_1) \le  f(x)+\dfrac{\epsilon}{\lambda}\|x-x_1\|$ for all $x \in \mathbb{R}^n.$
\end{enumerate}
\end{lemma}

\section{Normal cones at infinity}\label{Section3}

In this section, let $\Omega$ be a locally closed subset of $\mathbb{R}^n$ and $I$ be a nonempty subset of $\{1, \ldots, n\}.$ Consider the projection $\pi \colon \mathbb{R}^n \to \mathbb{R}^{\#I}, x := (x_1, \ldots, x_n) \mapsto (x_i)_{i \in I}$ and assume that the set $\pi(\Omega)$ is unbounded. In what follows, the notation $ \pi(x) \xrightarrow{\Omega} \infty $ means that $x \in \Omega$ and $\pi(x) \to \infty$. 

\begin{definition}{\rm 
The {\em normal cone to the set $\Omega$ at infinity (with respect to the index set $I$)} is defined by
\begin{eqnarray*}
N_{\Omega}(\infty_I) &:=& \Limsup_{\pi(x) \xrightarrow{\Omega} \infty} \widehat{N}_{\Omega}(x).
\end{eqnarray*} 
When $I = \{1, \ldots, n\},$ we write $N_{\Omega}(\infty)$ instead of $N_{\Omega}(\infty_I).$
}\end{definition}

\begin{remark}{\rm 
For a function $f \colon \mathbb{R}^n \to \overline{\mathbb{R}},$ the definition of $\partial f(\infty),$ which relates to the variation of $f(x)$ as $x$ tends to infinity, will be given in terms of the normal cone to the epigraph of $f$ at infinity. Hence, we will be only interested in points $(x, y) \in \mathrm{epi} f \subset \mathbb{R}^n \times \mathbb{R}$ as $x$ goes to infinity, i.e., points in the epigraph of $f$ whose $i$-th coordinates tend to infinity for some  $i$ in the index set $\{1, \ldots, n\} \subset \{1, \ldots, n, n + 1\}.$ This explains why the index set $I$ appears in the above definition of tangent and normal cones at infinity.
}\end{remark}

By definition, it is not hard to see that $N_{\Omega}(\infty_I)$ is a (not necessarily convex) closed cone.
\begin{example}\label{Example3.3} {\rm 
(i) Let $\Omega :=\left\{(x_1,x_2) \in \mathbb{R}^2 \mid x_1 x_2 = 1 \right\}.$ For $x := (x_1,x_2) \in \Omega$, we have 
$$\widehat{N}_{\Omega}(x) = N_{\Omega}(x) = \{(- tx_1^{-2}, -t) \mid t \geq 0\}.$$ Therefore,
$$N_{\Omega}(\infty_I) = 
\begin{cases}
\{0\}\times (-\mathbb{R}_{+}) & \textrm{ if } I = \{1\}, \\
(-\mathbb{R}_{+})\times \{0\} & \textrm{ if } I = \{2\},\\
\{0\}\times (-\mathbb{R}_{+}) \cup (-\mathbb{R}_{+})\times \{0\}  & \textrm{ if } I = \{1,2\}. 
\end{cases}$$

(ii) Let $\Omega :=\{(x_1,x_2) \in \mathbb{R}^2 \mid x_2 \geq e^{x_1}\}.$ For $x :=  (x_1,x_2) \in \Omega$, we have 
$$\widehat{N}_{\Omega}(x) =N_{\Omega}(x) = \begin{cases}
\{(0,0)\} & \textrm{ if } x_2 >e^{x_1}, \\
\{(te^{x_1}, -t) \mid t \geq 0\} & \textrm{ if } x_2=e^{x_1}.
\end{cases}$$ Therefore,
$$N_{\Omega}(\infty_I) = 
\begin{cases}
\{0\}\times (-\mathbb{R}_{+}) \cup \mathbb{R}_{+}\times \{0\} & \textrm{ if } I = \{1\}, \\
\mathbb{R}_{+}\times \{0\} & \textrm{ if } I = \{2\},\\
\{0\}\times (-\mathbb{R}_{+}) \cup \mathbb{R}_{+}\times \{0\} & \textrm{ if } I = \{1,2\}. 
\end{cases}$$
}\end{example}

\begin{remark}{\rm 
Recently, the notions of {\em Clarke tangent and normal cones at infinity} were introduced in \cite{PHAMTS2023-4}. Namely, 
the {\em Clarke tangent cone of $\Omega$ at infinity (with respect to the index set $I$)}, denoted by $T_\Omega(\infty_I),$  is the set of all vectors $v \in \mathbb{R}^n$ such that, whenever we have sequences $x_k \in \Omega$ with $\pi(x_k) \to \infty$ and $t_k \searrow 0,$ there exists a sequence $v_k \to v$ such that $x_k + t_k v_k \in \Omega$ for all $k;$ by the {\em normal cone to $\Omega$ at infinity (with respect to the index set $I$)}, we mean the {\em polar} of $T_\Omega(\infty_I),$ i.e., 
\begin{eqnarray*}
N_\Omega(\infty_I) &:=& \{\xi  \in \mathbb{R}^n \ | \ \langle \xi, v \rangle \le 0 \quad \textrm{ for all } \quad v \in T_\Omega(\infty_I)\}.
\end{eqnarray*}
We note that the notions of  Clarke and limiting normal cones at infinity are different, as seen from Example~\ref{Example3.3}(ii) and \cite[Example~4.2(ii)]{PHAMTS2023-4}; also see Remark~\ref{remark4.7} below.
}\end{remark}

\begin{proposition}\label{Proposition3.5}
The following equalities hold
\begin{eqnarray}
N_{\Omega}(\infty_I) &=& \Limsup_{\pi(x) \xrightarrow{\Omega} \infty} N_{\Omega}(x), \label{Eqn1} \\
 &=& \Limsup_{\pi(x) \rightarrow \infty} \Big[\mathrm{cone} \big(x - \Pi_{\Omega}(x) \big) \Big]. \label{Eqn2} 
\end{eqnarray}
Recall that $\Pi_{\Omega}(x) := \{y \in \Omega \ | \  \|x - y \| = d_{\Omega}(x)\}.$
\end{proposition}

\begin{proof}
To justify \eqref{Eqn1}, we note that $\widehat{N}_{\Omega}(x) \subset N_{\Omega}(x)$ for all $x \in \mathbb{R}^n.$ By definition, hence
\begin{eqnarray*}
N_{\Omega }(\infty_I)  &=& \Limsup_ {\pi(x) \xrightarrow{\Omega} \infty}  \widehat{N}_{\Omega}(x) \ \subset \ \Limsup_{\pi(x) \xrightarrow{\Omega} \infty}  N_{\Omega}(x).
\end{eqnarray*}
For the converse, take any $\xi \in \Limsup_{\pi(x) \xrightarrow{\Omega} \infty} N_{\Omega}(x);$ there exist sequences $x_k \in \Omega$ and $\xi_k \in N_{\Omega}(x_k) $ such that $\pi(x_k) \to \infty$ and $\xi_k \to \xi$. By definition,  for each $k > 0,$ there are $x_k' \in \Omega$ and $\xi_k' \in \widehat{N}_{\Omega}(x_k') $ such that $\|x_k' - x_k\| < \frac{1}{k}$ and $\|\xi_k' - \xi_k \| < \frac{1}{k}.$ Clearly, 
$\pi(x_k') \to \infty$ and $\xi_k' \to \xi$ as $k \to \infty.$ By definition, therefore $\xi \in {N}_{\Omega}(\infty_I).$

To verify~\eqref{Eqn2}, pick any $\xi \in N_{\Omega}(\infty_I).$ It follows from the equality~\eqref{Eqn1} that there exist sequences $x_k \in \Omega$ and $\xi_k \in N_{\Omega}(x_k) $ such that $\pi(x_k) \to \infty$ and $\xi_k \to \xi.$ Note that (see Lemma~\ref{Lemma2.3}) 
\begin{eqnarray*}
N_{\Omega}(x_k) &=& \Limsup_{x \rightarrow x_k} \Big[\mathrm{cone} \big(x - \Pi_{\Omega}(x) \big) \Big].
\end{eqnarray*}
Therefore, for each $k > 0$ there are  $x_k' \in \mathbb{R}^n,$ $y_k' \in \Pi_{\Omega}(x_k') $, and $t_k' \ge 0$ such that 
$$\|x_k' - x_k\| < \frac{1}{k} \quad \textrm{ and } \quad \| t_k'(x_k' - y_k') -\xi_k\| < \frac{1}{k},$$
which imply that $\pi(x_k') \to \infty$ and $t_k'(x_k' - y_k')\to \xi$ as $k \to \infty.$ By definition, we get $\xi \in \Limsup_{\pi(x) \rightarrow \infty} \Big[\mathrm{cone} \big(x - \Pi_{\Omega}(x) \big) \Big].$

To prove the opposite inclusion in~\eqref{Eqn2}, take any $\xi \in \Limsup_{\pi(x) \rightarrow \infty} \Big[\mathrm{cone} \big(x - \Pi_{\Omega}(x) \big) \Big]$. Then, there are sequences $x_k \in \mathbb{R}^n,$ $y_k \in \Pi_{\Omega}(x_k)$, and $t_k \ge 0$ such that $\pi(x_k) \to \infty$ and $\xi_k := t_k(x_k - y_k) \to \xi.$ By Lemma~\ref{Lemma2.3} again, $\xi_k \in N_{\Omega}(x_k)$. Letting $k \to \infty$ and  employing the equality~\eqref{Eqn1}, we obtain $\xi \in N_{\Omega}(\infty_I).$
\end{proof}

\begin{proposition}\label{Proposition3.6} 
The set $\pi(\partial \Omega)$ is unbounded if and only if $N_{\Omega}(\infty_{I})\neq \{0\}.$
\end{proposition}

\begin{proof}
{\em Necessity.} Assume that $\pi(\partial \Omega)$ is unbounded, i.e., there is a sequence $x_k \in \partial \Omega$ such that $\pi(x_k) \to \infty.$
In view of \cite[Proposition~1.2]{Mordukhovich2018}, for each $k$ there is a unit vector $\xi_k \in N_{\Omega}(x_k).$ The set of cluster points of the sequence $\xi_k$ is a nonempty subset of $\mathbb{S}^{n-1}$ and, by Proposition~\ref{Proposition3.5}, is contained in $N_{\Omega}(\infty_{I}).$
	
{\em Sufficiency.} Let $N_{\Omega}(\infty_{I})$ contain a unit vector $\xi.$
In view of Proposition~\ref{Proposition3.5}, there exist sequences $x_k \in \Omega$ with $\pi(x_k) \to \infty$ and $\xi_k \in N_{\Omega}(x_k)$ such that $\xi_k \to \xi$. Then $\xi_k \neq 0$ for all $k$ large enough. Applying \cite[Proposition~1.2]{Mordukhovich2018} again, we get $x_k \in \partial \Omega$ as required.
\end{proof}

\begin{proposition}
Let $\Omega_1, \Omega_2 \subset \mathbb{R}^n$ be locally closed sets satisfying the normal qualification condition at infinity
\begin{eqnarray*}
N_{\Omega_1}(\infty_I) \cap \big (-N_{\Omega_2}(\infty_I) \big) &=& \{0\}.
\end{eqnarray*}
Then
\begin{eqnarray*}
N_{\Omega_1 \cap \Omega_2}(\infty_I)  &\subset& N_{\Omega_1}(\infty_I) + N_{\Omega_2}(\infty_I).
\end{eqnarray*}
\end{proposition}

\begin{proof}
We first show that there exists a constant $R > 0$ such that
\begin{eqnarray*}
N_{\Omega_1}(x) \cap \big(-N_{\Omega_2}(x) \big) &=& \{0\}
\end{eqnarray*}
for all $x \in \Omega_1 \cap \Omega_2$ with $\|\pi(x)\| \ge R.$ Indeed, if this were not true, there would exist sequences $x_k \in \Omega_1 \cap \Omega_2$ and $\xi_k \in N_{\Omega_1}(x_k) \cap (-N_{\Omega_2}(x_k))$ such that $\pi(x_k) \to \infty$ and $\xi_k \ne 0.$ Since $N_{\Omega_1}(x_k)$ and $N_{\Omega_2}(x_k)$ are cones, we can assume that $\|\xi_k\| = 1$ for all $k.$ Passing to a subsequence if necessary we may assume that the sequence $\xi_k$ converges to some $\xi.$ Clearly, $\|\xi \| = 1$ and  $\xi \in N_{\Omega_1}(\infty_I) \cap \big(-N_{\Omega_2}(\infty_I) \big),$ in contradiction to our assumption.

We now prove the desired inclusion. To this end, take any $\xi \in N_{\Omega_1 \cap \Omega_2}(\infty_I).$ By Proposition~\ref{Proposition3.5}, there exist
sequences $x_k \in \Omega_1 \cap \Omega_2$ and $\xi_k \in N_{\Omega_1 \cap \Omega_2}(x_k)$ such that $\pi(x_k) \to \infty$ and $\xi_k \to \xi.$ Then for all $k$ sufficiently large, we have
\begin{eqnarray*}
N_{\Omega_1}(x_k) \cap \big(-N_{\Omega_2}(x_k) \big) &=& \{0\}.
\end{eqnarray*}
It follows from \cite[Theorem~2.16]{Mordukhovich2018} that
\begin{eqnarray*}
N_{\Omega_1 \cap \Omega_2}(x_k)  &\subset& N_{\Omega_1}(x_k) + N_{\Omega_2}(x_k).
\end{eqnarray*}
Consequently, we can write $\xi_k = u_k + v_k$ for some $u_k \in N_{\Omega_1}(x_k)$ and $v_k \in N_{\Omega_2}(x_k).$ There are two cases to be considered.

\subsubsection*{Case 1: the sequence $u_k$ is bounded.} 
Then the sequence $v_k$ is bounded too. Passing to subsequences if necessary we may assume that $u_k$ and $v_k$ converge to some $u$ and $v,$ respectively. Clearly, $\xi = u + v,$ and by Proposition~\ref{Proposition3.5}, we have $u \in N_{\Omega_1}(\infty_I) $ and $v \in N_{\Omega_2}(\infty_I).$ 
Therefore $\xi  \in  N_{\Omega_1}(\infty_I)  +  N_{\Omega_2}(\infty_I).$

\subsubsection*{Case 2: the sequence $u_k$ is unbounded.} 
Then the sequence $v_k$ is unbounded too. Since $u_k + v_k$ is a convergent sequence, it is easy to see that the sequence $\frac{u_k}{\|v_k\|}$ is bounded. Passing to subsequences if necessary we may assume that the sequences $\frac{u_k}{\|v_k\|}$ and $\frac{v_k}{\|v_k\|}$ converge to some $u$ and $v,$ respectively. Clearly, $u \in  N_{\Omega_1}(\infty_I), v \in N_{\Omega_2}(\infty_I)$ and $u = - v \ne 0,$ in contradiction to our assumption.
\end{proof}

\begin{proposition}
Let $\Omega_1$ and $\Omega_2$ be locally closed sets in $\mathbb{R}^n$ and $\mathbb{R}^m,$ respectively.
For two nonempty index sets $I_1 \subset \{1, \ldots, n\}$ and $I_2 \subset \{1, \ldots, m\},$ we have 
\begin{eqnarray*}
N_{\Omega_1}(\infty_{I_1})\times N_{\Omega_2}(\infty_{I_2}) &\subset& N_{\Omega_1 \times \Omega_2}(\infty_{I}),
\end{eqnarray*}
where $I := I_1 \cup \{n + i  \mid i \in I_2\}\subset \{ 1, \ldots,n+m\}$. 
\end{proposition}

\begin{proof}
In view of  \cite[Proposition~1.4]{Mordukhovich2018}, we have
\begin{eqnarray*}
N_{\Omega_1}(x_1) \times N_{\Omega_2}(x_2) &=& N_{\Omega_1 \times \Omega_2}(x_1, x_2) \quad \textrm{ for all } \quad (x_1, x_2) \in \Omega_1 \times \Omega_2,
\end{eqnarray*}
which, together with Proposition~\ref{Proposition3.5}, gives the desired inclusion.
\end{proof}

The inclusion in the above proposition may be strict. Indeed, let $n = m =1,$ $\Omega_1=\Omega_2=\mathbb{R}_{+}$, and $I_1=I_2=\{1\}, I = \{1, 2\}.$ Then it is not hard to check that
	\begin{eqnarray*}
	N_{\Omega_1}(\infty_{I_1})\times N_{\Omega_2}(\infty_{I_2})=\{(0,0)\} , \quad {\rm while} \quad  N_{\Omega_1 \times \Omega_2}(\infty_{I}) = \{0\}\times (-\mathbb{R}_{+}) \cup (-\mathbb{R}_{+})\times \{0\}.  
\end{eqnarray*}

\section{Subdifferentials at infinity} \label{Section4}

This section presents the notions of subdifferentials at infinity for extended real valued functions and then describes some of their fundamental properties.
Let $f \colon \mathbb{R}^n \to \overline{\mathbb{R}}$ be a lower semi-continuous function and let $I := \{1, \ldots, n\} \subset \{1, \ldots, n, n + 1\}.$ To avoid triviality, we will assume that $f$ is {\em proper at infinity} in the sense that the set $\mathrm{dom} f$ is unbounded.

\begin{definition}{\rm 
The {\em limiting and singular subdifferentials} of $f$ at infinity are defined by 
\begin{eqnarray*}
	\partial f(\infty) &:=& \{u \in \mathbb{R}^n \ | \ (u, -1) \in N_{\mathrm{epi} f}(\infty_{I})\},\\
	\partial^{\infty} f(\infty) &:=& \{u \in \mathbb{R}^n \ | \ (u, 0) \in N_{\mathrm{epi} f}(\infty_{I})\}. 
\end{eqnarray*}
}\end{definition}

It is easy to check that the limiting and singular subdifferentials at infinity are (not necessarily convex) closed sets. 
\begin{remark}\label{remark4.7}{\rm 
Recently, the notion of Clarke subdifferential at infinity was introduced in \cite{PHAMTS2023-4}. Namely, the {\em Clarke subdifferential of $f$ at infinity} is defined by
\begin{eqnarray*}
\partial^{\mathrm{Clarke}} f(\infty) &:=& \{u \in \mathbb{R}^n \ | \ (u, -1) \in N^{\mathrm{Clarke}}_{\textrm{epi} f}(\infty_I)\}. 
\end{eqnarray*} 
Notice that the Clarke and limiting subdifferentials at infinity may be different; for instance, for the function $f \colon \mathbb{R} \to \mathbb{R}, x \mapsto e^{x},$  we have
\begin{eqnarray*}
\partial^{\mathrm{Clarke}} f(\infty)\ = \mathbb{R}_+ & \supset & \partial f(\infty) \ = \ \{0\},
\end{eqnarray*}
 while for the function $f \colon \mathbb{R}^2 \to \mathbb{R}, (x_1, x_2) \mapsto x_1^3 + x_2,$ we have
 \begin{eqnarray*}
 \partial^{\mathrm{Clarke}} f(\infty)\ = \ \emptyset & \subset & \partial f(\infty) \ =\ \mathbb{R}_{+}\times \{1\}.
 \end{eqnarray*}
 Also note that if $f$ is {\em Lipschitz at infinity}, then $\partial^{\mathrm{Clarke}} f(\infty) = \mathrm{co} \, \partial f(\infty),$ as seen from  Proposition~\ref{Proposition5.6} below.
}\end{remark}

\begin{lemma}\label{Lemma4.3}
The following relationships hold
\begin{eqnarray*}
	\partial f(\infty) &=& \{u \in \mathbb{R}^n \ | \ (u, -1) \in \mathcal{N} \},\\
	\partial^{\infty} f(\infty) &=& \{u \in \mathbb{R}^n \ | \ (u, 0) \in \mathcal{N}\},
\end{eqnarray*}
where $\mathcal{N} := \displaystyle \Limsup_{x \to \infty} {N}_{\mathrm{epi} f}(x,f(x)).$
\end{lemma}	
	
\begin{proof} 
By definition, it suffices to verify that 
\begin{eqnarray*}
\Limsup_{x \to \infty} {N}_{\mathrm{epi} f}(x,f(x)) &=& \Limsup_{x \to \infty, \, r \geq f(x)} \widehat{N}_{\textrm{epi} f}(x,r).
\end{eqnarray*}
The inclusion $\subset$ is clear because we know from Proposition~\ref{Proposition3.5} that
\begin{eqnarray*}
\Limsup_{x \to \infty, \, r \geq f(x)} {N}_{\textrm{epi} f}(x,r) &=& \Limsup_{x \to \infty, \, r \geq f(x)} \widehat{N}_{\textrm{epi} f}(x,r).
\end{eqnarray*}
For the inclusion $\supset,$ we observe that for all $x \in \mathrm{dom} f$ and $r \geq f(x),$
\begin{eqnarray*}
T_{\textrm{epi} f}(x,f(x)) &\subset& T_{\textrm{epi} f}(x, r),
\end{eqnarray*}
which, together with the first equality in Lemma~\ref{Lemma2.3}, yields
\begin{eqnarray*}
\widehat{N}_{\textrm{epi} f}(x,f(x))  &\supset&  \widehat{N}_{\textrm{epi} f}(x, r).
\end{eqnarray*}
Therefore
\begin{eqnarray*}
\Limsup_{x \to \infty} {N}_{\textrm{epi} f}(x,f(x)) 
&\supset& \Limsup_{x \to \infty} \widehat{N}_{\textrm{epi} f}(x,f(x))  \\
&\supset& \Limsup_{x \to \infty, \, r \geq f(x)} \widehat{N}_{\textrm{epi} f}(x, r),
\end{eqnarray*}
as required.
\end{proof}

The following result will play a crucial role in later developments.

\begin{proposition}\label{Proposition4.4} 
We have the relationships
\begin{eqnarray}
\partial f(\infty) &=& \Limsup_{x \to \infty} \partial f(x) \ =\ \Limsup_{x \to \infty} \widehat{\partial}f(x), \label{Eqn3} \\
\partial^{\infty} f(\infty)  &=& \Limsup_{x \to \infty, r \searrow 0} r \partial f(x) \ \supseteq \ \Limsup_{x \to \infty} \partial^{\infty} f(x). \label{Eqn4}
\end{eqnarray}
\end{proposition}

\begin{proof}
To prove \eqref{Eqn3}, we need to show that
$$\Limsup_{x \to \infty} \widehat{\partial}f(x) \ \subset \ \Limsup_{x \to \infty} \partial f(x) \ \subset \ \partial f(\infty) \ \subset \ \Limsup_{x \to \infty} \widehat{\partial}f(x).$$
The first inclusion is clear. For the second inclusion, let $u \in \Limsup_{x \to \infty} \partial f(x).$ Then there are sequences $x_k \in \mathbb{R}^n$ and $u_k \in \partial f(x_k)$ satisfying $x_k \to \infty$ and $u_k \to u.$ Consequently, $(u_k,-1) \in N_{\textrm{epi}f}(x_k,f(x_k))$ and $(u_k, -1) \to (u, -1).$ By Lemma~\ref{Lemma4.3}, then $u \in  \partial f(\infty).$

For the third inclusion, take any $u \in \partial f(\infty).$ Applying Lemma~\ref{Lemma4.3}, we get sequences $x_k \in \mathbb{R}^n$ and $(u_k, -r_k) \in N_{\textrm{epi} f}(x_k,f(x_k)) $ such that $x_k \to \infty$ and $(u_k, -r_k) \to (u, -1).$ By the definition of normal cones, there are $x_k' \in \mathbb{R}^n$ and $(u_k', -r_k') \in \widehat{N}_{\textrm{epi} f}(x_k', f(x_k')) $ such that
\begin{eqnarray*}
\|x_k' - x_k\| < \frac{1}{k} \quad \textrm{ and } \quad \| (u_k', -r_k') - (u_k, -r_k) \| < \frac{1}{k}.
\end{eqnarray*}
Hence, $x_k' \to \infty$ and $(u_k', -r_k') \to (u, -1).$ For all $k$ sufficiently large, we have $r_k' > 0$ and so $\left(\dfrac{1}{r_k'}u_k', -1\right) \in \widehat{N}_{\textrm{epi} f}(x_k', f(x_k')),$ which yields $\dfrac{1}{r_k'} u_k '\in \widehat{\partial} f(x_k').$ Therefore 
$$u = \lim_{k \to \infty} \dfrac{1}{r_k'} u_k'  \in \Limsup_{x \to \infty} \widehat{\partial} f(x).$$

We next verify the equality in~\eqref{Eqn4}. Let $u \in \partial^{\infty} f(\infty).$ By Lemma~\ref{Lemma4.3}, there are sequences $x_k \to \infty$ and $(u_k, -r_k) \in N_{\textrm{epi} f}(x_k, f(x_k)) $ such that $(u_k, -r_k) \to (u, 0).$ Note that $r_k \ge 0$ (see \cite[Exercise~6.19]{Rockafellar1998}). Let $K:=\{ k \in \mathbb{N} \mid r_k = 0\}$. We have the following two cases:

\subsubsection*{Case 1: the set $K$ is finite} 
In this case, for all $k$ sufficiently large, we have $r_k > 0$ and so $\left(\dfrac{1}{r_k}u_k, -1 \right) \in N_{\textrm{epi} f}(x_k,f(x_k)).$ By definition, $\dfrac{1}{r_k}u_k \in \partial f(x_k).$ Therefore,
$$u = \lim_{k \to \infty} u_k = \lim_{k \to \infty} r_k \left(\dfrac{1}{r_k}u_k \right) \in \Limsup_{x \to \infty, r \searrow 0} r \partial f(x).$$ 

\subsubsection*{Case 2: the set $K$ is infinite} 
By passing to subsequences if necessary, we can suppose that $r_k = 0$ for all $k.$ By definition, $u_k \in \partial^{\infty} f(x_k)$ for all $k.$ It follows from Lemma~\ref{Lemma2.6} that for each $k > 0,$ there are $x_k' \in \mathbb{R}^n,$ $u_k' \in \widehat{\partial} f(x_{k}') \subseteq \partial f(x_{k}')$ and $0 < r_k' < \frac{1}{k}$ such that
\begin{eqnarray*}
\|x_k' - x_k\| < \frac{1}{k} \quad \textrm{ and } \quad \|r_k' u_k' - u_k\| < \frac{1}{k}.
\end{eqnarray*}
Therefore, $x_k' \to \infty$ and $r_k' u_k' \to u,$ which yield $u \in \Limsup_{x \to \infty, r \searrow 0} r \partial f(x).$

Conversely, take any $u \in \Limsup_{x \to \infty, r \searrow 0} r \partial f(x).$ There exist sequences $x_k \to \infty$, $r_k \searrow 0$, and $u_k \in \partial f(x_k)$ such that $r_k u_k \to u.$ We have $(u_k,-1)\in N_{\textrm{epi}f}(x_k,f(x_k)) $ and so $(r_k u_k, -r_k) \in N_{\textrm{epi}f}(x_k,f(x_k)).$ Since $(r_k u_k, -r_k) \to (u, 0)$ as $k \to \infty,$ we deduce from Lemma~\ref{Lemma4.3} that $u \in \partial^{\infty} f(\infty).$

Finally, for the inclusion in \eqref{Eqn4}, pick any $u \in \Limsup_{x \to \infty} \partial^{\infty} f(x)$. There are sequences $x_k \in \mathbb{R}^n$ and $u_k \in \partial^{\infty} f(x_k)$ such that $x_k \to \infty$ and $u_k \to u.$ It follows from Lemma~\ref{Lemma2.6} that for each $k > 0,$ there are $x_k' \in \mathbb{R}^n,$ $u_k' \in \widehat{\partial} f(x_{k}') \subseteq \partial f(x_{k}')$ and $0 < r_k' < \frac{1}{k}$ such that
\begin{eqnarray*}
\|x_k' - x_k\| < \frac{1}{k} \quad \textrm{ and } \quad \|r_k' u_k' - u_k\| < \frac{1}{k}.
\end{eqnarray*}
Therefore, $x_k' \to \infty$ and $r_k' u_k' \to u,$ which yield $u \in \Limsup_{x \to \infty, r \searrow 0} r \partial f(x).$
\end{proof}

\begin{example}\label{Example4.5}{\rm 
Consider an unbounded closed set $\Omega \subset \mathbb{R}^n.$ Then the indicator function $\delta_{\Omega}$ is lower semi-continuous. Moreover, in view of  Lemma~\ref{Lemma2.7}, we have for all $x \in \Omega$ that
$$\partial \delta_{\Omega} (x) =\partial^{\infty} \delta_{\Omega} (x)= N_{\Omega}(x).$$
It follows from Proposition~\ref{Proposition4.4} that
$$\partial \delta_{\Omega} (\infty)=\partial^{\infty} \delta_{\Omega} (\infty)=N_{\Omega}(\infty).$$
}\end{example}

\begin{remark}{\rm 
(i) In view of Proposition~\ref{Proposition4.4}, we have for any $\lambda > 0,$
\begin{align*}
\partial (\lambda f)(\infty) &= \Limsup_{x \to \infty} \partial (\lambda f)(x)= \Limsup_{x \to \infty} \lambda \partial f(x)= \lambda \Limsup_{x \to \infty} \partial f(x)=\lambda \partial f(\infty), \\
\partial^{\infty} (\lambda f)(\infty) &= \Limsup_{x \to \infty, r \searrow 0} r \partial (\lambda f)(x)= \Limsup_{x \to \infty, r \searrow 0} r \lambda \partial f(x)= \Limsup_{x \to \infty, r \searrow 0} r\partial f(x)=\partial^{\infty} f(\infty). 
\end{align*}
			
(ii) The inclusion in \eqref{Eqn4} may be strict. For instance, consider the smooth function $f \colon \mathbb{R} \to \mathbb{R}, x \mapsto x^3.$ We have for all $x\in \mathbb{R},$ 
$$\partial f(x) = \{3x^2\} \quad \textrm{ and } \quad \partial^{\infty} f(x)=\{0\} .$$ 
It follows from Proposition~\ref{Proposition4.4} that
$$\Limsup_{x \to \infty} \partial^{\infty} f(x) = \{0\} \quad \textrm{ and } \quad  \partial^{\infty} f(\infty) = \mathbb{R}_{+}.$$ 
Therefore, 
$\Limsup_{x \to \infty} \partial^{\infty} f(x)$ is a proper subset of $\partial^{\infty} f(\infty).$

(iii) The next example show that the limiting subdifferential at infinity of a smooth function may not be a singleton set and that the limiting subdifferential at infinity of a convex function is not necessarily a convex set.
}\end{remark}

\begin{example}\label{Example4.7}{\rm
(i) Consider the smooth function $f \colon \mathbb{R}^2 \to \mathbb{R}$,  $(x_1,x_2)\mapsto x_1^3+x_2$. We have $\partial f(x_1,x_2)=\{(3x_1^2,1)\}$ for all $(x_1,x_2)\in \mathbb{R}^2$. By virtue of Proposition~\ref{Proposition4.4}, $\partial f(\infty)=\mathbb{R}_{+}\times \{1\}$, which is not a singleton set.

(ii) Consider the convex function $f \colon \mathbb{R} \to \mathbb{R}$,  $x\mapsto f(x)=|x|$. We have $\partial f(x) = \{\mathrm{sign} \, x\}$ for all $x \ne 0.$ In view of Proposition~\ref{Proposition4.4}, $\partial f(\infty)=\{-1,1\}$, which is not a convex set.
}\end{example}

A crucial advantage of subdifferentials at infinity is the following non-emptiness property.

\begin{proposition}\label{Proposition4.8} 
We have
$$\partial f(\infty) \cup \left( \partial^{\infty} f(\infty) \setminus\{0\} \right) \neq \emptyset.$$	
\end{proposition} 

\begin{proof}
The epigraph of $f$ is closed because $f$ is lower semi-continuous. Moreover, $\pi(\partial\, {\mathrm{epi} f})$ is unbounded as ${\textrm{dom} f}$ is unbounded (by our assumption). Applying Proposition~\ref{Proposition3.6}, we obtain $N_{{\textrm{epi} f}}(\infty_{I})\neq \{0\}.$ 
The desired claim is obvious then from the definitions of limiting and singular subdifferentials at infinity.
\end{proof}

\begin{proposition}[sum rule] \label{Proposition4.9}
Let $f_1, f_2 \colon \mathbb{R}^n \to \overline{\mathbb{R}}$ be lower semi-continuous functions satisfying the (singular subdifferential) qualification condition at infinity
\begin{eqnarray}\label{Eqn5}
\partial^{\infty} f_1(\infty) \cap \big(- \partial^{\infty} f_2(\infty) \big) &=& \{0\}.
\end{eqnarray}
Then we have the relationships
\begin{eqnarray*}
\partial (f_1 + f_2)(\infty) &\subset& \partial f_1(\infty) + \partial  f_2(\infty), \\
\partial^{\infty} (f_1 + f_2)(\infty) &\subset& \partial^{\infty} f_1(\infty) + \partial^{\infty} f_2(\infty).   
\end{eqnarray*}
\end{proposition}

\begin{proof}
We first show that there exists a constant $R > 0$ such that
\begin{eqnarray}\label{Eqn6}
\partial^{\infty} f_1(x) \cap \big(- \partial^{\infty} f_2(x) \big) &=& \{0\} \quad \textrm{ when } \quad \|x\| > R.
\end{eqnarray}
Indeed, if this were not true, there would exist sequences $x_k \in \mathbb{R}^n$ and $w_k \in \partial^{\infty} f_1(x_k) \cap \big(-\partial^{\infty} f_2(x_k) \big)$ such that $x_k \to \infty$ and $w_k \ne 0.$ Replacing $w_k$ by $\frac{w_k}{\|w_k\|}$ we can assume that $\|w_k\| = 1$ for all $k.$ Then passing to a subsequence if necessary we may assume that the sequence $w_k$ converges to some $w.$ Clearly, $\|w \| = 1;$ moreover, in view of Proposition~\ref{Proposition4.4}, $w \in \partial^{\infty} f_1(\infty) \cap \big(- \partial^{\infty} f_2(\infty) \big),$ which contradicts our assumption.

We now prove the first inclusion; the proof of the second one is similar and so is omitted.
Take any $w \in \partial (f_1 + f_2)(\infty).$ By Proposition~\ref{Proposition4.4}, there exist
sequences $x_k \in \mathbb{R}^n$ and $w_k \in \partial (f_1 + f_2)(x_k)$ such that $x_k \to \infty$ and $w_k \to w.$ Putting 
the relation~\eqref{Eqn6} together with \cite[Theorem~2.19]{Mordukhovich2018}, we can write $w_k = u_k + v_k$ for some $u_k \in \partial f_1(x_k)$ and $v_k \in \partial f_2(x_k).$ There are two cases to be considered.

\subsubsection*{Case 1: the sequence $u_k$ is bounded.} 
Then the sequence $v_k$ is bounded too. Passing to subsequences if necessary we may assume that $u_k$ and $v_k$ converge to some $u$ and $v,$ respectively. By Proposition~\ref{Proposition4.4}, $u \in \partial f_1(\infty) $ and $v \in \partial f_2(\infty).$ Therefore $w  = u + v \in  \partial f_1(\infty) + \partial  f_2(\infty).$

\subsubsection*{Case 2: the sequence $u_k$ is unbounded.} 
Then the sequence $v_k$ is unbounded too. Since $u_k + v_k$ is a convergent sequence, it is easy to see that the sequence $\frac{u_k}{\|v_k\|}$ is bounded. Passing to subsequences if necessary we may assume that the sequences $\frac{u_k}{\|v_k\|}$ and $\frac{v_k}{\|v_k\|}$ converge to some $u$ and $v,$ respectively. Clearly, $\|v\| = 1$ and $u + v = 0;$ moreover, by Proposition~\ref{Proposition4.4}, $u \in \partial^\infty f_1(\infty) $ and $v \in \partial^\infty f_2(\infty).$ These contradict the condition~\eqref{Eqn5}.
\end{proof}

\begin{remark}{\rm 
(i) The condition~\eqref{Eqn5} in Proposition~\ref{Proposition4.9} can not be removed. Indeed, consider the functions $f_1(x) :=-x^2$ and $f_2(x) :=x^2.$ By Proposition~\ref{Proposition4.4}, it is easy to check that 
$$\partial^\infty f_1(\infty) = \partial^\infty f_2(\infty) = \mathbb{R},$$
and so the condition~\eqref{Eqn5} fails to hold. On the other hand, we have $\partial (f_1 + f_2)(\infty)=\{0\},$ while $\partial f_1(\infty) =	\partial f_2(\infty)=\emptyset.$ 

(ii) The condition~\eqref{Eqn5} holds automatically when either $f_1$ or $f_2$ is {\em Lipschitz at infinity}, see Proposition~\ref{Proposition5.2} below.
}\end{remark}

\begin{proposition}[subdifferentiation of maximum functions]
Let $f_1, f_2 \colon \mathbb{R}^n \to \overline{\mathbb{R}}$ be lower semi-continuous functions satisfying the qualification condition~\eqref{Eqn5}. Then the following inclusions are valid
\begin{eqnarray*}
\partial \left( {\rm max}\{f_1,f_2\}\right)(\infty) & \subset & \bigcup_{
\lambda_{1}, \lambda_2 \in [0,1],  \ \lambda_{1}+\lambda_{2} = 1} \{ \lambda_1 \circ \partial f_1(\infty) + \lambda_2 \circ \partial f_2 (\infty)\},\\
\partial^{\infty} \left( {\rm max}\{f_1,f_2\}\right)(\infty) & \subset &  \partial^{\infty} f_1(\infty) + \partial^{\infty} f_2 (\infty),
\end{eqnarray*}
where $\lambda \circ \partial f(\infty)$ is defined as $\lambda \partial f(\infty)$ if $\lambda>0$ and as $\partial^{\infty} f(\infty)$ if $\lambda=0.$
\end{proposition}

\begin{proof} 
We confine ourselves to proving the first inclusion; the proof of the second inclusion proceeds on the same lines. 

Pick any $w \in \partial \left( {\rm max}\{f_1,f_2\}\right)(\infty).$ By Proposition~\ref{Proposition4.4}, there exist sequences $x_k \in \mathbb{R}^n$ and $w_k \in \partial \left( {\rm max}\{f_1,f_2\}\right)(x_k)$ such that $x_k \to \infty$ and $w_k \to w$. As in the proof of Proposition~\ref{Proposition4.9}, the qualification condition~\eqref{Eqn5} gives for all $k$ large enough,
\begin{eqnarray*}
\partial^{\infty} f_1(x_k) \cap \big(- \partial^{\infty} f_2(x_k) \big) &=& \{0\},
\end{eqnarray*}
which, together with \cite[Theorem~4.10]{Mordukhovich2018}, yields
$$w_k \in \lambda_{1k} \circ \partial f_1(x_k) + \lambda_{2k} \circ \partial f_2 (x_k),$$
for some $\lambda_{1k}, \lambda_{2k} \in [0,1]$ with $\lambda_{1k}+\lambda_{2k}=1.$ Passing to subsequences if necessary, we can assume that the sequences $\lambda_{1k}$ and $\lambda_{2k}$ converge to $\lambda_1$ and $\lambda_2.$ Clearly, $\lambda_1, \lambda_2 \in [0, 1]$ and $\lambda_1 + \lambda_2 = 1.$ There are three cases to be considered.

\subsubsection*{Case 1: $\lambda_1 \in (0,1)$.} Then $\lambda_{2} \in (0, 1).$ For all $k$ sufficiently large, we have 
$\lambda_{1k} > 0$ and $\lambda_{2k} > 0,$ and so $w_k = \lambda_{1k}u_k+\lambda_{2k} v_k$ for some $u_k \in \partial f_1(x_k)$ and $v_k \in \partial f_2(x_k).$ Analysis similar to that in the proof of Proposition~\ref{Proposition4.9} shows that the sequences $u_k$ and $v_k$ are bounded. 
Passing to subsequences if necessary, we can assume that the sequences $u_k$ and $v_k$ converge to $u$ and $v,$ respectively. Then $w = 
\lambda_1 u + \lambda_2 v  \in \lambda_1 \partial f_1(\infty) + \lambda_2  \partial f_2 (\infty),$ as required.

\subsubsection*{Case 2: $\lambda_{1}  = 0$} Then $\lambda_{2}  = 1.$ Set $K :=\{k \in \mathbb{N}\mid \lambda_{1k}=0 \}.$ 

We first suppose that the set $K$ is finite. Then for all $k$ sufficiently large, there are $u_k \in \partial f_1(x_k)$ and $v_k \in \partial f_2({x_k})$ such that $w_k=\lambda_{1k}u_k+\lambda_{2k} v_k.$ By the same argument in the proof of Proposition~\ref{Proposition4.9} (Case~2), it is easy to see that 
the sequence $v_k$ must be bounded. Passing to subsequences if necessary, we can assume that the sequence $v_k$ converges to some $v.$ Then $v  \in  \partial f_2(\infty)$ and the sequence $\lambda_{1k} u_k$ must converge to $w - v.$ As $\lambda_{1k} \to 0$ and $u_k \in \partial f_1(x_k),$ it follows from Proposition~\ref{Proposition4.4} that $w - v \in \partial^{\infty}f_1(\infty),$ and hence $w \in \partial^{\infty}f_1(\infty) + \partial f_2(\infty).$

We now assume that the set $K$ is infinite. Then for all $k \in K,$ we can write $w_k = u'_k + \lambda_{2k} v_k,$ where $u'_k \in \partial^{\infty}f_1(x_k)$ and  $v_k \in \partial f_2(x_k).$ As in the proof of Proposition~\ref{Proposition4.9}, it is not hard to see that the sequences $u_k'$ and $v_k$ are bounded.  Passing to subsequences if necessary, we can assume that the sequences $u_k'$ and $v_k$ converge to $u'$ and $v,$ respectively. Then $w = u' + \lambda v  \in \partial^\infty f_1(\infty) + \lambda_2  \partial f_2 (\infty),$ as required.

\subsubsection*{Case 3: $\lambda_{1}  = 1$} The proof is similar to Case 2 and is omitted.
\end{proof}

\begin{proposition}[subdifferentiation of minimum functions]
Let $f_1, f_2 \colon \mathbb{R}^n \to \overline{\mathbb{R}}$ be lower semi-continuous functions. Then
\begin{eqnarray*}
\partial \left( {\rm min}\{f_1,f_2\}\right)(\infty) & \subset & \partial f_1(\infty) \cup \partial f_2 (\infty),\\
\partial^{\infty} \left( {\rm min}\{f_1,f_2\}\right)(\infty) & \subset & \partial^{\infty} f_1(\infty) \cup \partial^{\infty} f_2 (\infty). 
	\end{eqnarray*}
\end{proposition}

\begin{proof}
We only verify the first inclusion because the proof of the second one is similar.
Take any $w \in \partial \left( {\rm min}\{f_1,f_2\}\right)(\infty).$ In light of Proposition~\ref{Proposition4.4}, there are sequences $x_k \to \infty$ and $w_k \in \partial \left( {\rm min}\{f_1,f_2\}\right)(x_k)$ satisfying $w_k \to w.$ By \cite[Proposition~4.9]{Mordukhovich2018}, $w_k \in \partial f_1(x_k) \cup \partial f_2 (x_k).$ Applying Proposition~\ref{Proposition4.4} again, we obtain $w \in \partial f_1(\infty) \cup \partial f_2 (\infty).$
\end{proof}

Another application of Proposition~\ref{Proposition4.4} handles {\em partial subdifferentiation}.

\begin{proposition}[partial subdifferentials]
For a proper lower semi continuous function 
$$\varphi \colon \mathbb{R}^n \times \mathbb{R}^m \to \overline{\mathbb{R}}, \quad (x, y) \mapsto \varphi (x,y),$$
and a point $\overline{y} \in \mathbb{R}^m,$ let $\partial_{x}\varphi (\infty, \overline{y})$ and $\partial_{x}^{\infty}\varphi (\infty,\overline{y})$ denote the limiting and singular subdifferentials of the function $\varphi(\cdot, \overline{y}) \colon \mathbb{R}^n \to \overline{\mathbb{R}}, x \mapsto \varphi(x, \overline{y})$ at infinity. Under the condition that $(0, v) \in \partial^{\infty}\varphi(\infty)$ implies that $v = 0,$ we have
\begin{eqnarray*}
	\partial_{x}\varphi (\infty,\overline{y}) & \subset & \{ u \in \mathbb{R}^n \mid \exists v \in \mathbb{R}^m \; \textrm{with}\; (u, v)\in \partial \varphi(\infty) \},\\
	\partial_{x}^{\infty}\varphi (\infty,\overline{y}) & \subset & \{ u \in \mathbb{R}^n \mid \exists v \in \mathbb{R}^m \; \textrm{with}\; (u, v)\in \partial^{\infty} \varphi(\infty) \}. 
\end{eqnarray*}
\end{proposition}

\begin{proof}
We only prove the first inclusion because the proof of the second one is similar. We first prove that there is a constant $R>0$ such that for all $x \in \mathbb{R}^n \setminus \mathbb{B}_{R}$, 
\begin{eqnarray}\label{Eqn7}
		(0, v) \in \partial \varphi(x,\overline{y}) \quad \textrm{ implies that } \quad v = 0. 
\end{eqnarray}
Indeed, if this were not true, there would exist sequences $x_k \in \mathbb{R}^n$ and $v_k \in \mathbb{R}^m$, $v_k \neq 0$ such that $x_k \to \infty$ and $(0,  v_k)\in \partial \varphi(x_k,\overline{y})$. Hence, $\left(0,\dfrac{v_k}{\|v_k\|}\right)\in \dfrac{1}{\|v_k\|}\partial \varphi(x_k, \overline{y}).$ Letting $k \to \infty$ and applying Proposition~\ref{Proposition4.4}, we get a unit vector $v$ such that $(0, v)\in \partial^{\infty} \varphi(\infty),$ a contradiction. 
	
We now take any $u \in \partial_{x}\varphi (\infty, \overline{y}).$ By Proposition~\ref{Proposition4.4}, there are sequences $x_k \to \infty$ and $u_k \in \partial_{x}\varphi (x_k,\overline{y})$ such that $u_k \to u$. For all $k$ sufficiently large, the condition~\eqref{Eqn7} holds at $x = x_k.$ It follows from \cite[Corollary~10.11]{Rockafellar1998} that
	$$
	\partial_{x}\varphi (x_k,\overline{y})\subset \{u \in \mathbb{R}^m \mid \exists v \in \mathbb{R}^n \; \textrm{with} \; (u, v) \in \partial \varphi (x_k,\overline{y})\}. 
	$$
Hence, there is $v_k \in \mathbb{R}^m$ such that $(u_k, v_k) \in \partial \varphi (x_k,\overline{y}).$
There are two cases:
	
\subsubsection*{Case 1: the sequence $y_k$ is bounded.}  We can assume that $v_k \to v$. As $x_k \to \infty$, we get $(x_k, v) \to \infty$ and $(u, v) \in \partial \varphi(\infty)$ as required.
	
\subsubsection*{Case 2: the sequence $v_k$ is unbounded.}  We have $\left(\dfrac{u_k}{\|v_k\|},\dfrac{v_k}{\|v_k\|} \right) \in \dfrac{1}{\|v_k\|} \partial \varphi (x_k,\overline{y}).$ Letting $k \to \infty$ and applying Proposition~\ref{Proposition4.4}, we get a unit vector $v$ such that $(0, v) \in \partial^{\infty} \varphi (\infty),$ which contradicts our assumption. 
\end{proof}

\section{Lipschitzness at infinity}\label{Section5}

In this section, we provide some properties of the Lipschitz continuity at infinity for lower semi-continuous functions. So, let $f \colon \mathbb{R}^n \to {\mathbb{R}}$ be a lower semi-continuous function, fix $I := \{1, \ldots, n\} \subset \{1, \ldots, n, n + 1\}$ and consider the projection $\pi \colon \mathbb{R}^n \times \mathbb{R} \to \mathbb{R}^{n}, (x, y) \mapsto x.$

\begin{definition}{\rm
We say that $f$ is {\em Lipschitz at infinity} if there exist constants $L > 0$ and $R > 0$ such that
\begin{eqnarray*}
\|f(x) - f(x') \| &\le& L \|x - x'\| \quad \textrm{ for all } \quad x, x' \in \mathbb{R}^n \setminus \B_R.
\end{eqnarray*}
}\end{definition}

\begin{proposition}\label{Proposition5.2}   
The function $f$ is Lipschitz at infinity if and only if $\partial^{\infty}f(\infty)=\{0\}.$ In this case, $\partial f(\infty)$ is nonempty compact. 
\end{proposition}

\begin{proof} 
We first suppose that $f$ is Lipschitz at infinity with constant $L>0$. Take any $u \in \partial^{\infty}f(\infty)$. In view of Proposition~\ref{Proposition4.4}, there are sequences $r_k \searrow 0$, $x_k \to \infty$, and $u_k \in \partial f(x_k)$ such that $r_k u_k \to u.$ For $k$ sufficiently large, $f$ is Lipschitz around $x_k$ with constant $L.$  By \cite[Theorem~1.22]{Mordukhovich2018}, we have $\|u_k \| \leq L$. Hence, $ r_k u_k \to 0$ and so $u = 0.$ Therefore $\partial^{\infty}f(\infty)=\{0\}.$ (Note that $0 \in \partial^{\infty}f(\infty).$)  

Conversely, assume that $\partial^{\infty}f(\infty) = \{0\}.$ Then there are constants $L > 0$ and $R > 0$ such that for all $x \in \mathbb{R}^n \setminus \B_{R}$ and all $u \in \partial f(x),$ it holds that $\|u\| \leq L.$ Indeed, if this were not true, there would exist sequences $x_k$ tending to infinity and $u_k \in \partial f(x_k)$ such that $\|u_k\| >  k$ for all $k.$ Passing to a subsequence if necessary we can assume that $\dfrac{1}{\|u_k\|}u_k \to u $ with $\|u\|=1$. According to Proposition~\ref{Proposition4.4}, we arrive at a contradiction that  $u \in \partial^{\infty}f(\infty)\setminus \{0\}.$ 
Therefore, for all $x \in \mathbb{R}^n \setminus \B_{R}$ and all $u \in \partial f(x),$ we have $\|u\| \leq L.$ In view of \cite[Theorem 4.15]{Mordukhovich2018}, $f$ is locally Lipschitz continuous on $\mathbb{R}^n \setminus \B_{R}$ with constant $L.$

On the other hand, there is a $C^{\infty}$-function $\varphi \colon \mathbb{R}^n \to [0, 1]$ such that 
$$\varphi(x) = \begin{cases}
1 & \textrm{ if } \ \|x\| \ge 3R, \\
0 & \textrm{ if } \ \|x\| \le 2R.
\end{cases}$$
Define the function $\overline{f} \colon \mathbb{R}^n \to \mathbb{R}$ by $\overline{f}(x) := f(x) \varphi(x).$ Then $\overline{f}  \equiv f$ on 
$\mathbb{R}^n\setminus \mathbb{B}_{3R}.$ Furthermore, it is not hard to see that $\overline{f}$ is locally Lipschitz and so it is globally Lipschitz on the compact set $\mathbb{B}_{3R}.$ Increasing $L$ if necessary, we may assume that $\overline{f}$ is locally Lipschitz with constant $L.$
It follows from \cite[Theorem 1.22]{Mordukhovich2018} that $\|u\| \leq L$ for all $u \in \partial \overline{f}(x)$ and all $x \in \mathbb{R}^n.$
Applying the mean value theorem (see \cite[Corollary 4.14]{Mordukhovich2018}), we get for all $x, x' \in \mathbb{R}^n$ that 
\begin{eqnarray*}
|\overline{f}(x) - \overline{f}(x') | &\leq& \|x-x'\| \sup \{ \|u\| \mid u \in \partial \overline{f}(y), y \in [x, x'] \} \ \leq \ L \|x-x'\|,
\end{eqnarray*}
where $[x, x']$ denotes the segment with endpoints $x$ and $x',$ i.e., $[x, x'] := \{(1 - t) x + tx'  \mid 0 \le t\le 1\}.$
In particular, if $x, x' \in \mathbb{R}^n \setminus \B_{3R}$ then 
\begin{eqnarray*}
|f(x) - f(x')| &=& |\overline{f}(x) - \overline{f}(x')| \  \le \ L\|x - x'\|.
\end{eqnarray*}
Therefore $f$ is Lipschitz at infinity.

The non-emptiness of $\partial f(\infty)$ follows directly from Proposition~\ref{Proposition4.8}. Pick any $u \in \partial f(\infty)$, there are $x_k \to \infty$ and $u_k \in \partial f(x_k)$ with $u_k \to u$. For $k$ sufficiently large, the Lipschitzness at infinity of $f$ ensures that $\|u_k\| \le L,$ and so $\|u\|\leq L.$ Therefore, $\partial f(\infty) \subseteq \mathbb{B}_{L}.$ Finally, $\partial f(\infty)$ is closed by definition, hence it is compact.	
\end{proof}

\begin{example}{\rm 
(i) The function 
$$f \colon \mathbb{R}^n \to \mathbb{R}, \quad x \mapsto \sqrt{\sum_{i = 1}^n |x_i|},$$ 
is globally Lipschitz on $\mathbb{R}^n \setminus \mathbb{B}_R$ for any $R > 0,$ and so it is Lipschitz at infinity.

(ii) Let $f \colon \mathbb{R}^n \to {\mathbb{R}}$ be a {\em piecewise linear function}, i.e., $\mathbb{R}^n = \cup_{i= 1}^k D_i$ with $D_i$ 
polyhedral, and for each $i$ there exist $a_i \in \mathbb{R}^n$ and $b_i \in \mathbb{R}$ such that 
\begin{eqnarray*}
f(x) &=&  \langle a_i, x \rangle +b_i \quad \textrm{ for all } \quad  x \in D_i. 
\end{eqnarray*}
The sets $D_i$ are closed, and $f$ is continuous relative to them, so because only finitely many are involved, $f$ is lower semi-continuous.

Set $I :=\{ i \mid D_i \; {\rm is \; unboubded}\}.$ Then there is a constant $R > 0$ such that for all $x \in \mathbb{R}^n  \setminus \mathbb{B}_{R}$ we have
\begin{eqnarray*}
f(x) &\leq& \max_{i \in I} f_i(x), 
\end{eqnarray*}	
which, together with Lemma~\ref{Lemma2.9}, implies that
\begin{eqnarray*}
\partial f(x) &\subset& \partial \max_{i \in I} f_i(x) \ \subset \ \textrm{co}\left(\bigcup_{i\in I}\{a_i\} \right).
\end{eqnarray*}
It follows from Proposition~\ref{Proposition4.4} that $\partial f(\infty) \subset \textrm{co}\left(\bigcup_{i\in I}\{a_i\} \right)$ and $\partial^{\infty} f(\infty)=\{0\},$ and so $f$ is Lipschitz at infinity (by Propositions~\ref{Proposition5.2}).
}\end{example}

\begin{remark}{\rm
The non-emptiness and compactness of $\partial f(\infty)$ do not yield the Lipschitzness at infinity of $f.$ For instance, consider the smooth function $f \colon \mathbb{R} \to \mathbb{R}, x \mapsto e^x.$ It follows from Proposition~\ref{Proposition4.4} that $\partial f(\infty)=\{0\}$ and $\partial^{\infty} f(\infty)=[0, \infty).$ However, by Proposition~\ref{Proposition5.2}, $f$ is not Lipschitz at infinity.
}\end{remark}

\begin{example}[subdifferentiation of distance functions]{\rm 
Let $\Omega \subset \mathbb{R}^n$ be a nonempty closed set. Clearly, the distance function
$$d_{\Omega} \colon \mathbb{R}^n \to \mathbb{R}, \quad x \mapsto \inf_{y \in \Omega} \|x - y\|,$$ 
is globally Lipschitz with constant $1.$ By Proposition~\ref{Proposition5.2}, $\partial^{\infty} d_{\Omega} (\infty) = \{0\}.$ Moreover, in view of  \cite[Theorem~1.33]{Mordukhovich2018}, we have 
\begin{eqnarray*}
\partial d_{\Omega}(x) &=& 
\begin{cases}
N_{\Omega}(x) \cap \mathbb{B} & \quad \textrm{ if } x \in \Omega, \\
\frac{x - \Pi_{\Omega}(x)}{d_{\Omega}(x)} & \quad \textrm{ otherwise,}
\end{cases}
\end{eqnarray*}
where $\Pi_{\Omega}(x) := \{y \in \Omega \ | \  \|x - y \| = d_{\Omega}(x)\}.$ We will show that
\begin{eqnarray*}
\partial d_{\Omega} (\infty) &=&
\begin{cases}
(N_{\Omega}(\infty) \cap \mathbb{B}) \cup \mathscr{D} &\textrm{ if  } \Omega \textrm{ is unbounded},\\
\mathbb{S}^{n - 1} &\textrm{ otherwise},
\end{cases}
\end{eqnarray*}
where $\mathbb{S}^{n - 1}$ is the unit sphere in $\mathbb{R}^n$ and $\mathscr{D} := \Limsup_{x \to \infty, \ x \not \in \Omega} \frac{x - \Pi_{\Omega}(x)}{d_{\Omega}(x)}.$ Indeed, in the light of Proposition~\ref{Proposition4.4}, it suffices to prove that if $\Omega$ is bounded, then
$\mathbb{S}^{n - 1} \subset \partial d_{\Omega} (\infty).$ To this end, take any $u \in \mathbb{S}^{n - 1}.$ By the classical Weierstrass theorem, there exists a point $\overline{x} \in \Omega$ such that 
\begin{eqnarray*}
\langle u, \overline{x} \rangle &\ge& \langle u, x \rangle  \quad \textrm{ for all } \quad x \in \Omega.
\end{eqnarray*}
It follows that for all $t > 0,$ we have 
$$\overline{x} + tu \not \in \Omega, \quad d_\Omega(\overline{x} + tu) = \|\overline{x} + tu - \overline{x}\| = t, \quad \textrm{ and } \quad \overline{x} \in \Pi_{\Omega}(\overline{x} + tu).$$ 
Hence
\begin{eqnarray*}
u &=& \frac{(\overline{x} + tu) - \overline{x}}{t} \ \in \ \ \partial d_{\Omega}(\overline{x} + tu).
\end{eqnarray*}
Letting $t \to \infty$ and applying Proposition~\ref{Proposition4.4}, we get $u \in \partial d_{\Omega} (\infty).$ Since $u$ was arbitrary in $\mathbb{S}^{n - 1},$ we conclude that $\mathbb{S}^{n - 1} \subset \partial d_{\Omega} (\infty),$ as required.
}\end{example}

\begin{proposition}\label{Proposition5.6} 
If $f$ is Lipschitz at infinity, then
\begin{eqnarray*}
{\rm co}\partial f(\infty) &=& \partial^{\mathrm{Clarke}} f(\infty). 
\end{eqnarray*}
\end{proposition}

\begin{proof} 
Since $f$ is Lipschitz at infinity, there exist constants $L > 0$ and $R > 0$ such that
\begin{eqnarray*}
\|f(x) - f(x') \| &\le& L \|x - x'\| \quad \textrm{ for all } \quad x, x' \in \mathbb{R}^n \setminus \B_R.
\end{eqnarray*}
Let $\Omega_f \subset \mathbb{R}^n$ be the set of points where $f$ is not differentiable. By Rademacher's theorem (see, for example, \cite{Rockafellar1998}), the set $\Omega_f \setminus \B_R$ has measure zero in $\mathbb{R}^n.$ Moreover, we have
$\|\nabla f(x) \| \le L$ for all $x \in \mathbb{R}^n \setminus (\Omega \cup \mathbb{B}_R).$ Let 
\begin{eqnarray*}
A &:=& \{ \lim \nabla f(x) \ | \ x \to \infty \textrm{ and } x\not \in \Omega_f\}.
\end{eqnarray*}
From \cite[Proposition~5.3]{PHAMTS2023-4} we know that $\partial^{\mathrm{Clarke}}f(\infty) = \mathrm{co} A.$ Hence, it remains only to verify that 
$\mathrm{co} \partial f(\infty) = \mathrm{co} A,$ which follows directly from the following two claims.

\subsubsection*{Claim 1: $\partial f(\infty) \subset \mathrm{co} A$} 
Since the function $f$ is Lipschitz on $\mathbb{R}^n \setminus \B_R,$ we know that $\partial f(x) \subset \partial^{\mathrm{Clarke}} f(x)$ for all $x \in \mathbb{R}^n \setminus \B_R$ (see, for example, \cite[Exercise~1.79]{Mordukhovich2018}). It follows from Proposition~\ref{Proposition4.4} that
\begin{eqnarray*}
	\partial f(\infty) &=& \Limsup_{x \to \infty} \partial f(x) \ \subset \ \Limsup_{x \to \infty} \partial^{\mathrm{Clarke}} f(x). 
\end{eqnarray*}
Take any $u \in \partial f(\infty)$. There are sequences $x_k$ and $u_k \in \partial^{\mathrm{Clarke}} f(x_k) $ such that $x_k \to \infty$ and $u_k \to u.$ By \cite[Theorem~2.5.1]{Clarke1990} and 
 the Carath\'eodory theorem (see \cite[Theorem~2.29]{Rockafellar1970}), for each $k>0$, we find $\lambda_{ik} \in [0,1]$ and $x_{ik}' \in \mathbb{R}^n \setminus \Omega_f$, for $i=1,\ldots,n+1$ such that $\sum_{i = 1}^{n+1}\lambda_{ik}=1,$ $\|x_{ik}' - x_k\| < \frac{1}{k}$ and
\begin{eqnarray*}
\left\|\sum_{i = 1}^{n+1} \lambda_{ik} \nabla f(x_{ik}')-u_k \right\| &<& \frac{1}{k}.
\end{eqnarray*}
Then for all $i = 1, \ldots, n + 1,$ we have $x_{ik}' \to \infty$ as $k \to \infty,$ and so $\|\nabla f(x_{ik}')\| \le L$ for all $k$ sufficiently large. Passing to subsequences, we may assume that $\nabla f(x_{ik}') \to u_i$, $\lambda_{ik} \to \lambda_{i} \in [0,1]$ with $\sum_{i = 1}^{n+1}\lambda_{i}=1$. Then $u=\sum_{i = 1}^{n+1} \lambda_{i} u_i$, and so $u \in {\rm co}A$. Therefore, $\partial f(\infty) \subset \mathrm{co}A.$

\subsubsection*{Claim 2: $A \subset \partial f(\infty)$} Pick any $u \in A.$ Then there is a sequence $x_k \to \infty$ with $x_k \notin \Omega_f$ such that $\nabla f(x_k) \to u.$ Note that $\widehat{\partial}f(x_k)=\{\nabla f(x_k)\}$ (see, for example, \cite[Exercise~8.8]{Rockafellar1998}). In view of Proposition~\ref{Proposition4.4}, $u \in \partial f(\infty).$ Thus $A \subset \partial f(\infty),$ and so the claim follows.
\end{proof}

\begin{proposition}[chain rule]\label{Proposition5.7} 
Let $f \colon \mathbb{R}^n \to {\mathbb{R}}$ be a lower semi-continuous function and let $g_1, \ldots, g_n \colon \mathbb{R}^m \to \mathbb{R}$ be Lipschitz at infinity such that the map $g \colon \mathbb{R}^m \to \mathbb{R}^m, x \mapsto (g_1(x), \ldots, g_n(x)),$ is coercive (i.e., $g(x) \to \infty$ as $x \to \infty).$ Assume that the following condition holds
$$\Big[u := (u_1,\ldots, u_n) \in \partial^{\infty}f(\infty) \; {\rm and} \; 0 \in \sum_{i = 1}^{n}|u_i|\partial \left( (\mathrm{sign}\, u_i) g_i \right) (\infty) \Big] \implies u = 0.$$
Then we have 
\begin{eqnarray*}
\partial (f \circ g)(\infty) & \subset & \bigcup_{(u_1, \ldots, u_n) \in \partial f(\infty)} \sum_{i = 1}^n 
|u_i| \partial \big( (\mathrm{sign}\, u_i) g_i \big) (\infty),\\
\partial^{\infty} (f \circ g)(\infty) & \subset & \bigcup_{(u_1, \ldots, u_n) \in \partial f^{\infty}(\infty)} \sum_{i = 1}^n 
|u_i| \partial \big( (\mathrm{sign}\, u_i) g_i \big) (\infty). 
\end{eqnarray*}
\end{proposition}

\begin{proof} 
We first claim that there is a constant $R > 0$ such that for all $x \in \mathbb{R}^n \setminus \B_R,$ the following condition satisfies:
\begin{eqnarray}\label{Eqn8}
\Big[y \in \partial^{\infty}f(g(x)) \ \textrm{ and } \ 0 \in \partial \langle y, g \rangle (x) \Big] \quad & \Longrightarrow&  \quad y = 0. 
\end{eqnarray}
If this were not true, there would exist sequences $x_k \to \infty$ and $u_k \neq 0$ satisfying 
$u_k \in \partial^{\infty}f(g(x_k))$ and $0 \in \partial \langle u_k, g \rangle (x_k)$.  By definition, $\dfrac{u_k}{\|u_k\|} \in \partial^{\infty}f(g(x_k))$. Passing to a subsequence and noting that $g(x_k) \to \infty$ as $x_k \to \infty,$ we may assume that the sequence $\dfrac{u_k}{\|u_k\|}$ converges to some $u^* := (u^*_1, \ldots, u^*_n).$ By Proposition~\ref{Proposition4.4}, $u^* \in \partial^{\infty}f(\infty)\setminus\{0\}.$

On the other hand, our assumptions ensure that the functions $g_1, \ldots, g_n$ are Lipschitz on $\mathbb{R}^n \setminus \B_R$ (perhaps after increasing $R$). It follows from Lemma~\ref{Lemma2.9} that for all $k$ sufficiently large,
\begin{eqnarray*}
0 \ \in \ \partial \langle u_k, g \rangle (x_k) &  = & \partial \left( \sum_{i = 1}^{n}u_{ik}g_i  \right)(x_k) \\
& \subset & \sum_{i = 1}^{n} \partial (u_{ik}g_i)(x_k) \ = \ \sum_{i = 1}^{n}|u_{ik}|\partial \left( (\mathrm{sign}\, u_{ik}) g_i \right) (x_k),
\end{eqnarray*}
where $u_k := (u_{1k}, \ldots, u_{nk}) \in \mathbb{R}^n$ for all $k.$ By dividing both sides of these by $\|u_k\|$ and letting $k \to \infty,$ we get
\begin{eqnarray*}
0 &\in & \sum_{i = 1}^{n}|u^*_i|\partial \left( (\mathrm{sign}\, u^*_i) g_i \right) (\infty),
\end{eqnarray*}
which contradicts our assumption.

Now, invoking  Lemma~\ref{Lemma2.10} and taking \eqref{Eqn8} into account, we have for all $x \in \mathbb{R}^n \setminus \mathbb{B}_R,$
\begin{eqnarray*}
\partial \big( f \circ g \big) (x)
&\subset & \bigcup_{(u_1, \ldots, u_n)  \in \partial f (g(x))} \partial \left( \sum_{i = 1}^n u_i g_i \right) (x) \\
&\subset & \bigcup_{(u_1, \ldots, u_n)  \in \partial f (g(x))} \sum_{i = 1}^n  \partial \big( u_i g_i \big) (x) \\
&\subset & \bigcup_{(u_1, \ldots, u_n)  \in \partial f (g(x))} \sum_{i = 1}^n  |u_i| \partial \big(  (\mathrm{sign}\, u_i)  g_i \big) (x).
\end{eqnarray*}
Similarly, we also have 
\begin{eqnarray*}
\partial^\infty \big( f \circ g \big) (x)
&\subset & \bigcup_{(u_1, \ldots, u_n)  \in \partial^\infty f (g(x))} \sum_{i = 1}^n  |u_i| \partial \big(  (\mathrm{sign}\, u_i)  g_i \big) (x).
\end{eqnarray*}
Letting $x \to \infty$ and applying Proposition~\ref{Proposition4.4}, we get the desired inclusions.
\end{proof}

The coerciveness of the map $g$ in Proposition~\ref{Proposition5.7} is essential as seen from the  following example.

\begin{example}{\rm 
Let $f, g \colon \mathbb{R} \to \mathbb{R}$ be functions given by 
\begin{eqnarray*}
f(y) := 
\begin{cases}
-\ln y &  \textrm{ if } y > 0, \\
0 & \textrm{ otherwise,} 
\end{cases}
\quad \textrm{ and } \quad
g(x) := 
\begin{cases}
e^{-x} &  \textrm{ if } x > 0, \\
-1 & \textrm{ otherwise.} 
\end{cases}
\end{eqnarray*}
We have $\partial f(\infty) = \{0\}$ while $\partial (f \circ g)(\infty) = \{0, 1\};$ the function $g$ is Lipschitz at infinity but not coercive. 
}\end{example}

\begin{proposition}\label{Proposition5.9} 
Let $g_i, h_j\colon \mathbb{R}^n \to \mathbb{R}, i = 1, \ldots, p, \ j = 1, \ldots q$ be Lipschitz at infinity such that the set
\begin{eqnarray*}
\Omega &:=& \{x \in \mathbb{R}^n \mid g_i(x)\leq 0, i=1,\ldots,p, h_j(x)=0,j =1,\ldots,q\}
\end{eqnarray*}
is nonempty and unbounded. If the following constraint qualification at infinity holds
\begin{equation}\label{Eqn9}
0 \ \notin \ {\rm co}\left\{\bigcup_{i=1}^p\partial g_i(\infty),\bigcup_{j=1}^q \Big( \partial h_j(\infty) \cup \partial (-h_j)(\infty) \Big) \right\},
\end{equation} 
then
\begin{equation*} 
N_{\Omega}(\infty) \ \subset \ {\rm pos}\left\{\bigcup_{i=1}^p\partial g_i(\infty),\bigcup_{j=1}^q \Big( \partial h_j(\infty) \cup \partial (-h_j)(\infty) \Big) \right\}.
\end{equation*}	
\end{proposition}

\begin{proof}
Take any $u \in  N_{\Omega}(\infty).$ Suppose $u \ne 0,$ since everything is trivial otherwise. By Proposition~\ref{Proposition3.5}, there are sequences $x_k \in \Omega$  and $u_k \in N_{\Omega}(x_k)$ such that $x_k \to \infty$ and $u_k \to u.$ In view of \cite[Corollary~4.36]{Mordukhovich2006}, we have 
\begin{eqnarray*}
N_{\Omega}(x_k) & \subset & {\rm pos}\left\{\bigcup_{i=1}^p \partial g_i(x_k),\bigcup_{j=1}^q \Big( \partial h_j(x_k) \cup \partial (-h_j)(x_k) \Big) \right\}.
\end{eqnarray*}
Hence, there exist sequences $\lambda_{ik} \geq 0$, $v_{ik} \in \partial g_i(x_k) $ for $i=1,\ldots,p$, $\mu_{jk}\geq 0$, and $w_{jk} \in  \partial h_j(x_k) \cup \partial (-h_j)(x_k)$ for  $j=1,\ldots,q$ such that
\begin{eqnarray}\label{Eqn10}
u_k &=&  \sum_{i=1}^{p}\lambda_{ik}v_{ik}+\sum_{j=1}^{q}\mu_{jk}w_{jk}. 
\end{eqnarray}
For all $k$ sufficiently large, we have $u_k \neq 0$ and so $s_k:=\sum_{i=1}^{p}\lambda_{ik}+\sum_{j=1}^{q}\mu_{jk} > 0.$ By dividing both sides of the above equality, we get
\begin{eqnarray*}
\dfrac{u_k}{s_k} &=& \sum_{i=1}^{p}\dfrac{\lambda_{ik}}{s_k}v_{ik}+\sum_{j=1}^{q}\dfrac{\mu_{jk}}{s_k}w_{jk}. 
\end{eqnarray*}
Note that $\dfrac{\lambda_{ik}}{s_k}, \dfrac{\mu_{jk}}{s_k} \in [0,1]$. Passing to subsequences if necessary, we can assume that $\dfrac{\lambda_{ik}}{s_k} \to \lambda_{i} \geq 0$ and $\dfrac{\mu_{jk}}{s_k} \to \mu_j \geq 0.$ Clearly,  $\sum_{i=1}^{p}\lambda_{i}+\sum_{j=1}^{q}\mu_{j}=1.$ Since $g_i$ and $h_j$ are Lipschitz at infinity, 
the sequences $v_{ik}$ and $w_{jk}$ are bounded and so we may assume that they converge to some $v_i$ and $w_j.$ Certainly 
$v_i \in \partial g_i(\infty)$ and $w_j \in \big(\partial h_j(\infty) \cup \partial (-h_j)(\infty)\big)$ (in view of Proposition~\ref{Proposition4.4}). There are two cases to be considered.

\subsubsection*{Case 1: the sequence $s_k$ is bounded} We can assume that $s_k \to s \geq 0$. If $s = 0,$ then we can see from \eqref{Eqn10} that $u = 0,$ a contradiction. Thus $s > 0$ and so
\begin{eqnarray*}
\dfrac{u}{s} &=& \sum_{i=1}^{p}\lambda_{i}v_{i}+\sum_{j=1}^{q}\mu_j w_{j},
\end{eqnarray*}
as required.

\subsubsection*{Case 2: the sequence $s_k$ is unbounded} 
We have $\dfrac{u_k}{s_k} \to 0$ and so
\begin{eqnarray*}
0 &=& \sum_{i=1}^{p}\lambda_{i}v_{i}+\sum_{j=1}^{q}\mu_j w_{j},
\end{eqnarray*}
in contradiction to our constraint qualification. 
\end{proof}

\section{Applications in optimization}\label{Section6}

In this section, let $f \colon \mathbb{R}^n \to \overline{\mathbb{R}}$ be a lower semi-continuous function and let $\Omega \subset \mathbb{R}^n$ be a closed set such that the following conditions hold:
\begin{enumerate}[({\rm A}1)]
	\item $\mathrm{dom} f \cap \Omega$ is unbounded.
	\item $\partial^{\infty}f(\infty) \cap \big(-N_{\Omega}(\infty) \big) =\{0\}.$
	\item $f$ is bounded from below on $\Omega,$ i.e., $f_* := \inf_{x\in \Omega} f(x)$ is finite.
\end{enumerate}
Consider the problem
\begin{equation} \label{Problem}
\mathrm{minimize } \ f (x) \quad \textrm{ over } \quad x \in \Omega, \tag{P}
\end{equation}
and let $\mathrm{Sol}$ denote its optimal solution set.

\subsection{Optimality conditions at infinity}

The following is a necessary optimality condition at infinity to optimization problems (see also \cite{PHAMTS2023-4, PHAMTS2019-1}).

\begin{theorem}[first order condition for optimality] \label{Theorem6.1} 
	If $f$ does not attain its infimum on $\Omega$ then 
	\begin{eqnarray*}
		0 &\in& \partial f(\infty) + N_{\Omega}(\infty).
	\end{eqnarray*}
\end{theorem}

\begin{proof} 
	We begin with the following case.
	\subsubsection*{Case 1: $\Omega$ is the whole space $\mathbb{R}^n$}
	In this case, $N_{\Omega}(\infty) = \{0\}.$ Since $f$ does not attain its infimum on $\mathbb{R}^n,$ there exists a sequence $x_k \in \mathbb{R}^n$ tending to infinity such that
	\begin{eqnarray*}
		f(x_k) &\le& \inf_{x \in \mathbb{R}^n} f(x) +\dfrac{1}{k^2}.
	\end{eqnarray*}
	Employing the Ekeland variational principle (Lemma~\ref{Lemma2.11}), we find $\overline{x}_k \in \mathbb{R}^n$ for $k > 0$ such that
	\begin{eqnarray*}
		\|x_k - \overline{x}_k\| & \le & \dfrac{1}{k}, \\
		f(\overline{x}_k) &\le& f(x)+\dfrac{1}{k}\|x-\overline{x}_k\| \quad \textrm{for all } \quad x \in \mathbb{R}^n.
	\end{eqnarray*}
	The first inequality implies that the sequence $\overline{x}_k$ tends to infinity as the sequence $x_k$ tends to infinity, while the second inequality says that $\overline{x}_k$ is a global minimizer of the lower semi-continuous function
	$$\mathbb{R}^n \to \overline{\mathbb{R}} \quad x \mapsto f(x)+\dfrac{1}{k}\|x-\overline{x}_k\|.$$ 
	By the Fermat rule (Lemma~\ref{Lemma2.8}), we obtain
	\begin{eqnarray*}
		0 &\in& \partial \left( f+ \dfrac{1}{k}\|\cdot-\overline{x}_k\|\right)(\overline{x}_k).
	\end{eqnarray*}
	Since the function $\|\cdot-\overline{x}_k\|$ is 1-Lipschitz, the sum rule (see \cite[Theorem~2.19]{Mordukhovich2018}) gives
	\begin{eqnarray*}
		0 &\in& \partial f(\overline{x}_k)+ \dfrac{1}{k}\partial (\|\cdot-\overline{x}_k\|)(\overline{x}_k).
	\end{eqnarray*}
	Note that $\partial (\|\cdot-\overline{x}_k\|)(\overline{x}_k)=\mathbb{B}$. Hence, 
	$0 \in \partial f(\overline{x}_k)+ \dfrac{1}{k}\mathbb{B},$
	and so there is ${u}_k \in \partial f(\overline{x}_k)$ such that $\|{u}_k\| \leq \dfrac{1}{k}$. Letting $k \to \infty$ and applying Proposition~\ref{Proposition4.4}, we obtain  $ 0\in \partial f(\infty).$
	
	\subsubsection*{Case 2: $\Omega$ is an arbitrary subset of $\mathbb{R}^n$}
	We have
	\begin{eqnarray*}
		\inf_{x \in \mathbb{R}^n} \left(f + \delta_{\Omega}\right)(x) & = & \inf_{x \in \Omega} f(x) \ > \ -\infty,
	\end{eqnarray*}
	where $\delta_{\Omega} \colon \mathbb{R}^n \to \overline{\mathbb{R}}$ stands for the indicator function of the set $\Omega.$ Note that the (lower semi-continuous) function $f + \delta_{\Omega}$ does not attain its infimum on $\mathbb{R}^n.$ Therefore, $0 \in \partial (f + \delta_{\Omega})(\infty)$ (by the argument employed in the first case). This, together with Proposition~\ref{Proposition4.9}, Example~\ref{Example4.5} and the assumption~(A2), yields $0 \in \partial f (\infty) + N_{\Omega}(\infty).$
\end{proof}

\begin{example}{\rm 
(i) Consider the smooth function $f \colon \mathbb{R}^2 \to \mathbb{R}$,  $(x_1,x_2) \mapsto e^{x_1}+x_2^2$. Clearly, $\inf_{(x_1, x_2) \in \mathbb{R}^2} f(x_1, x_2) = 0$ and $f$ does not attain its infimum. For each integer $k,$ we have $\nabla f(-k, 0) = (e^{-k}, 0).$ Letting $k \to \infty,$ we get $(0,0) \in \partial f(\infty).$
		
(ii) Consider the smooth function $f \colon \mathbb{R}^2 \to \mathbb{R}, (x_1, x_2) \mapsto (x_1 x_2 - 1)^2 + x_1^2.$ Clearly, 
$$\inf_{(x_1, x_2) \in \mathbb{R}^2} f(x_1, x_2) = 0$$ 
and $f$ does not attain its infimum. For each integer $k,$ we have $\nabla f(\frac{1}{k}, k) = (\frac{2}{k}, 0).$ Letting $k \to \infty,$ we obtain $(0,0) \in \partial f(\infty).$

(iii) Let $f(x_1, x_2) := x_2$ and $\Omega := \{(x_1,x_2)\in \mathbb{R}^2 \mid  x_1^2x_2 \ge  1 \}.$ Then it is easy to see that $\inf_{(x_1, x_2) \in \Omega} f(x_1, x_2) = 0$ and $f$ does not attain its infimum on $\Omega.$ Moreover, we have
		\begin{eqnarray*}
			\partial f(\infty) &=& \{(0,1)\}, \\ 
			\partial^{\infty} f(\infty) &=& \{(0,0)\}, \\
			N_\Omega(\infty) &=& \{0\}\times (-\mathbb{R}_{+}) \cup \mathbb{R}\times \{0\}.
		\end{eqnarray*}
		Hence, the condition $\partial^{\infty}f(\infty)\cap (-N_{\Omega}(\infty))=\{0\}$ is fulfilled and $(0, 0) \in \partial f(\infty) + N_\Omega(\infty)$. 
}\end{example}

When Theorem~\ref{Theorem6.1} is combined with Proposition~\ref{Proposition5.9} we obtain a powerful Lagrange multiplier rule at infinity.

\begin{corollary}[Lagrange multipliers]
Consider the problem~\eqref{Problem} where the constraint set $\Omega$ is given by
\begin{eqnarray*}
\Omega &:=& \{x \in \mathbb{R}^n \mid g_i(x)\leq 0, i=1,\ldots,p, h_j(x)=0,j =1,\ldots,q\}
\end{eqnarray*}
with $g_i, h_j\colon \mathbb{R}^n \to \mathbb{R}, i = 1, \ldots, p, \ j = 1, \ldots q$ being Lipschitz at infinity such that 
the constraint qualification at infinity~\eqref{Eqn9} holds. If $f$ does not attain its infimum on $\Omega$ then there exist  non-negative real numbers $\lambda_i, i = 1, \ldots, p$ and $\mu_j, j = 1, \ldots, q$ such that 
	\begin{equation*}
	0 \in  \partial f(\infty) + \sum_{i=1}^p \lambda_i \partial g_i(\infty) + \sum_{j=1}^q \mu_j \Big( \partial h_j(\infty) \cup \partial (-h_j)(\infty) \Big). \end{equation*}
\end{corollary}

\subsection{Coercivity and weak sharp minima at infinity}

The next theorem presents some properties of optimization problems.

\begin{theorem}[coercivity and weak sharp minima] \label{Theorem6.4}
	Assume that $0 \not \in \partial f(\infty)+ N_{\Omega}(\infty).$ Then the following statements hold:
	\begin{enumerate}[{\rm (i)}]
		\item The optimal solution set $\mathrm{Sol}$ in the problem \eqref{Problem} is nonempty compact.
		\item There exist constants $c > 0$ and $R > 0$ such that
		\begin{eqnarray*}
			f(x) - f_* &\ge& c\, \mathrm{dist} (x, \mathrm{Sol}) \quad \textrm{ for all } \quad x \in \Omega \setminus \mathbb{B}_R,
		\end{eqnarray*}
		where $\mathrm{dist} (x, \mathrm{Sol})$ stands for the distance from $x$ to $\mathrm{Sol}.$
		\item $f$ is coercive on $\Omega,$ i.e., if $x \to \infty$ and $x \in \Omega$ then $f(x) \to \infty.$
	\end{enumerate}
\end{theorem}

\begin{proof}
	The assumption (A2), together with Example~\ref{Example4.5}, implies that
	\begin{eqnarray*}
		\partial^{\infty}f(\infty)\cap \big(-\partial^\infty \delta_{\Omega}(\infty) \big) &=& \{0\}.
	\end{eqnarray*}
	Hence, from the proof of Proposition~\ref{Proposition4.9}, we can find a constant $R > 0$ such that
	\begin{eqnarray} \label{Eqn11}
	\partial^{\infty}f(x)\cap \big(-\partial^\infty \delta_{\Omega}(x) \big) &=& \{0\} \quad \textrm{ when} \quad \|x\| > R.
	\end{eqnarray}
	
	(i) The set $\mathrm{Sol}$ is closed because $f$ is lower semi-continuous and $\Omega \subset \mathbb{R}^n$ is closed. Moreover, by Theorem~\ref{Theorem6.1}, it is nonempty. Hence, it remains to show that $\mathrm{Sol}$ is bounded. By contradiction, suppose that there exists a sequence $x_k \in \mathrm{Sol}$ tending to infinity. Observe that
	\begin{eqnarray*}
		f_* &:=& \inf_{x\in \Omega} f(x) \ = \ \inf_{x\in \mathbb{R}^n} (f + \delta_\Omega)(x).
	\end{eqnarray*}
	Hence, $x_k$ is a global minimizer of the lower semi-continuous function $f + \delta_\Omega \colon \mathbb{R}^n \to \overline{\mathbb{R}}.$ By the Fermat rule (Lemma~\ref{Lemma2.8}), we obtain $0 \in \partial (f + \delta_\Omega)(x_k).$ 
	For $k$ sufficiently large, the relation~\eqref{Eqn11} holds at $x_k$ and so the sum rule (see \cite[Theorem~2.19]{Mordukhovich2018}) can be applied. Therefore,
	\begin{eqnarray*}
		0 &\in& \partial f(x_k) +  \partial \delta_\Omega(x_k) \ = \ \partial f(x_k) + N_\Omega(x_k),
	\end{eqnarray*}
	Letting $k \to \infty$ we deduce from Propositions~\ref{Proposition3.5} and \ref{Proposition4.4} that $0 \in \partial f(\infty)+ N_{\Omega}(\infty),$ which contradicts the assumption~(A2).
	
	(ii) By contradiction, suppose that there exists a sequence $x_k \in \Omega$ tending to infinity such that
	\begin{eqnarray*}
		0 \ \le \ f(x_k) - f_* &<& \frac{1}{k^2}\, \mathrm{dist} (x_k, \mathrm{Sol}).
	\end{eqnarray*}
	Then $\mathrm{dist} (x_k, \mathrm{Sol}) > 0.$ Applying the Ekeland variational principle (Lemma~\ref{Lemma2.11}), we get a sequence $\overline{x}_k \in \Omega$ such that 
	\begin{eqnarray*}
		\|x_k - \overline{x}_k\| &\leq& \dfrac{1}{k}\mathrm{dist} (x_k, \mathrm{Sol}), \\
		f(\overline{x}_k) - f_* &\leq& f(x) - f_* + \frac{1}{k} \|x - \overline{x}_k\| \quad \textrm{for all } \quad x \in \Omega.
	\end{eqnarray*}
	The first inequality implies that the sequence $\overline{x}_k$ tends to infinity as the sequence $x_k$ tends to infinity and the set $\mathrm{Sol}$ is bounded; while the second inequality says that $\overline{x}_k$ is a global minimizer of the lower semi-continuous function
	$$\mathbb{R}^n \to \overline{\mathbb{R}}, \quad x \mapsto f(x) - f_* + \delta_\Omega(x) + \dfrac{1}{k}\|x -\overline{x}_k\| .$$
	By the Fermat rule (Lemma~\ref{Lemma2.8}), we get
	$$0 \in \partial \big ( f(\cdot) - f_* + \delta_\Omega(\cdot)+\dfrac{1}{k}\|\cdot-\overline{x}_k\| \big)(\overline{x}_k). $$
	The sum rule (see \cite[Theorem~2.19]{Mordukhovich2018}) gives
	\begin{eqnarray*}
		0
		&\in& \partial \big ( f(\cdot) - f_* + \delta_\Omega(\cdot) \big)(\overline{x}_k) + \dfrac{1}{k}\partial (\|\cdot-\overline{x}_k\|)(\overline{x}_k) \\
		&\subset& \partial f(\overline{x}_k) +\partial \delta_{\Omega}(\overline{x}_k)+ \dfrac{1}{k}\partial (\|\cdot-\overline{x}_k\|)(\overline{x}_k).
	\end{eqnarray*}
	Note that $\partial \delta_{\Omega}(\overline{x}_k)=N_\Omega(\overline{x}_k)$ and 
	$\partial (\|\cdot-\overline{x}_k\|)(\overline{x}_k)=\mathbb{B}.$
	Therefore, 
	$$0 \in \partial f(\overline{x}_k)+N_\Omega(\overline{x}_k)+ \dfrac{1}{k} \mathbb{B},$$ 
	and so there exist ${u}_k \in \partial f(\overline{x}_k)$ and ${v}_k \in N_\Omega(\overline{x}_k)$ such that 
	\begin{equation*}
	\|{u}_k + {v}_k\| \leq \dfrac{1}{k}.
	\end{equation*}
	There are two cases to be considered:
	\subsubsection*{Case 1: the sequence ${u}_k$ is bounded} 
	Then the sequence ${v}_k$ is bounded too. Passing to subsequences, we may assume that ${u}_k$ and ${v}_k$ converge to some $u$ and $v$, respectively. 
	Clearly, $u + v = 0;$ moreover, by Propositions~\ref{Proposition3.5} and \ref{Proposition4.4}, $u \in \partial f(\infty)$ and $v \in N_\Omega(\infty).$ Hence, $0 \in \partial f(\infty)+N_\Omega(\infty),$ which contradicts our assumption.
	
	\subsubsection*{Case 2: the sequence ${u}_k$ is unbounded}
	Then the sequence ${v}_k$ is unbounded too, and so we may assume that $\|v_k\| \to \infty.$ Since $u_k + v_k$ is a convergent sequence, it is easy to see that the sequence $\frac{u_k}{\|v_k\|}$ is bounded. Passing to subsequences if necessary we may assume that the sequences $\frac{u_k}{\|v_k\|}$ and $\frac{v_k}{\|v_k\|}$ converge to some $u$ and $v,$ respectively. 
	Clearly, $u = -v \ne 0;$ moreover, by Proposition~\ref{Proposition4.4}, $u \in \partial^\infty f(\infty) $ and $v \in \partial^\infty \delta(\infty) = N_{\Omega}(\infty).$ Therefore $\partial^{\infty}f(\infty)\cap \big(-N_{\Omega}(\infty) \big) \ne \{0\},$ which contradicts the assumption~(A2).
	
	(iii) This is a direct consequence of (i) and (ii).
\end{proof}

\subsection{Stability}
We present now a stability theorem about parametrized optimization problems, which comes out of the results already obtained.
To this end, for each $u \in \mathbb{R}^n$ define the function $f_u \colon \mathbb{R}^n \to \overline{\mathbb{R}}$ by $f_u(x) := f(x) - \langle u, x \rangle.$ Consider the perturbed optimization problem
\begin{equation} \label{Pu}
\mathrm{minimize } \ f_u(x) \quad \textrm{ over } \quad x \in \Omega, \tag{P$_u$}
\end{equation}
and let $\mathrm{Sol}(u)$ denote its optimal solution set. 

\begin{theorem}[stability results]
	Assume that $0 \not \in \partial f(\infty)+ N_{\Omega}(\infty).$ Then there exists a constant $\epsilon > 0$ such that for all $u \in \mathbb{B}_\epsilon,$ the following statements hold:
	\begin{enumerate}[{\rm (i)}]
		\item $f_u$ is bounded from below on $\Omega.$ 
		\item The optimal solution set $\mathrm{Sol}(u)$ in the problem~\eqref{Pu} is nonempty compact.
		\item The inclusion
		\begin{eqnarray*}
			\Limsup_{u \to 0} \mathrm{Sol}(u) &\subset& \mathrm{Sol}(0)
		\end{eqnarray*}
		is valid
	\end{enumerate}
\end{theorem}

\begin{proof}
	The desired statements will be established after some preliminaries. We first show that
	\begin{eqnarray*}
		\epsilon_1 &:=& \inf\{\|u\| \mid u \in \partial (f + \delta_\Omega)(\infty)\} \ > \ 0.
	\end{eqnarray*}
	By contradiction, suppose that there exists a sequence $u_k \in \partial (f + \delta_\Omega)(\infty)$ such that $u_k \to 0.$ In view of Proposition~\ref{Proposition4.4}, we find sequences $x_k \to \infty$ and $v_k \in \partial (f + \delta_\Omega)(x_k)$ such that $\|u_k - v_k\| < \frac{1}{k}.$ Then
	\begin{eqnarray*}
		0 \ = \ \lim_{k \to \infty} u_k \ = \ \lim_{k \to \infty} v_k \ \in \  \Limsup_{x \to \infty} \partial (f + \delta_\Omega)(x) 
		&= & \partial (f + \delta_\Omega)(\infty).
	\end{eqnarray*}
	It follows from Proposition~\ref{Proposition4.9} and Example~\ref{Example4.5} that
	\begin{eqnarray*}
		0 &\in& \partial f (\infty) + \partial \delta_\Omega(\infty) \ = \ \partial f(\infty)  + N_\Omega(\infty) ,
	\end{eqnarray*}
	in contradiction to our assumption.
	
	We next show that there exists a constant $R_1 > 0$ such that
	\begin{eqnarray} \label{Eqn12}
	u &\not \in& \partial (f + \delta_\Omega)(x) \quad \textrm{ for all } \quad \|u\| < \epsilon_1 \quad \textrm{ and all } \quad \|x\| < R_1.
	\end{eqnarray}
	Indeed, if this is not the case, then there exist a vector $u \in \mathbb{R}^n$ with $\|u\| < \epsilon_1$ and a sequence $x_k \to \infty$ such that $u \in \partial (f + \delta_\Omega)(x_k)$ for all $k.$ This, together with Proposition~\ref{Proposition4.4}, yields $u \in \partial (f + \delta_\Omega)(\infty),$ which contradicts the definition of $\epsilon_1.$
	
	Finally, by Theorem~\ref{Theorem6.4}, there exist constants $c > 0$ and $R_2 > 0$ such that the optimal solution set $\mathrm{Sol}(0)$ in the problem~\eqref{Problem} is nonempty compact and is contained in $\mathbb{B}_{R_2}$ and that
	\begin{eqnarray*}
		f(x) - f_* &\ge& c\, \mathrm{dist} (x, \mathrm{Sol}(0)) \quad \textrm{ for all } \quad x \in \Omega \setminus \mathbb{B}_{R_2},
	\end{eqnarray*}
	where $\mathrm{dist} (x, \mathrm{Sol}(0))$ stands for the distance from $x$ to $\mathrm{Sol}(0).$
	
	Let $\epsilon \in (0, \min\{\epsilon_1, \frac{c}{2}\}),$ $R := \max\{R_1, R_2\}$ and take any $u \in \mathbb{B}_\epsilon.$ We now show the desired statements. 
	
	(i) We have for all $x \in \Omega \setminus \mathbb{B}_{2R},$ 
	\begin{eqnarray*}
		f_u(x) 
		&=& f(x) - \langle u, x \rangle \\
		&\ge& f(x) - \epsilon \|x\| \\
		&\ge& f_* + c\, \mathrm{dist} (x, \mathrm{Sol}(0)) - \epsilon \|x\| \\
		&=& f_* + c\, \|x - \overline{x}\| - \epsilon \|x\| \\
		&\ge& f_* + (c - \epsilon) \|x\| - c\, \|\overline{x}\| \\
		&\ge& f_* + \frac{c}{2} \|x\| - cR \\
		&\ge& f_*,
	\end{eqnarray*}
	where $\overline{x}$ is a point in $\mathrm{Sol}(0)$ such that $\|x - \overline{x}\| = \mathrm{dist} (x, \mathrm{Sol}(0)).$ 
	Therefore, $f_u$ is bounded from below on $\Omega$ because we know from the assumption~(A3) that $f$ is bounded from below on $\Omega.$
	
	(ii) By the assumption~(A1), $\mathrm{dom} f_u \cap \Omega$ is unbounded. In light of Proposition~\ref{Proposition4.9}, it is easy to see that $\partial^\infty f_u(\infty) \subset \partial^\infty f(\infty)$ for all $u \in \mathbb{R}^n.$
	This, together with the assumption~(A2), yields
	\begin{eqnarray*}
		\partial^\infty f_u (\infty) \cap \big( -N_{\Omega} (\infty) \big)  & = & \{0\}.
	\end{eqnarray*}
	Therefore, the assumptions~(A1)--(A3) hold with $f$ replaced by $f_u.$ Moreover, the choice of $\epsilon,$ 
	along with the inclusion $\partial f_u(\infty) + N_{\Omega}(\infty) \subset \partial f(\infty) + N_{\Omega}(\infty) - u,$
	ensures that
	\begin{eqnarray*}
		0 & \not \in & \partial f_u(\infty) + N_{\Omega}(\infty).
	\end{eqnarray*}
	On the other hand, it is clear that $\mathrm{Sol}(u)$ is the optimal solution set in the problem
	\begin{equation*}
	\mathrm{minimize } \ \big( f_u + \delta_{\Omega} \big ) (x) \quad \textrm{ over } \quad x \in \mathbb{R}^n.
	\end{equation*}
	Hence, by Theorem~\ref{Theorem6.4}, the set $\mathrm{Sol}(u)$ is nonempty compact. Moreover it is contained in $\mathbb{B}_R$ 
	because otherwise there is $x \in \mathrm{Sol}(u)$ with $\|x\| > R;$ then, in light of the Fermat rule (Lemma~\ref{Lemma2.8}), we get
	\begin{eqnarray*}
		0 &\in& \partial (f_u + \delta_{\Omega})(x) \ = \ \partial (f + \delta_{\Omega})(x)  - u
	\end{eqnarray*}
	which contradicts \eqref{Eqn12}.
	
	(iii) Take any $\overline{x} \in \mathrm{Sol}(0).$ By definition, we have for all $x \in \mathrm{Sol}(u)$ that
	\begin{eqnarray*}
		f(\overline{x})  - \langle u, x \rangle & \le & f(x) - \langle u, x \rangle  \ = \ f_u(x) \ \le \ f_u(\overline{x})   \ = \ f(\overline{x})  - \langle u, \overline{x} \rangle.
	\end{eqnarray*}
	Therefore,
	\begin{eqnarray*}
		\inf_{x \in \Omega} f(x) \ = \ f(\overline{x}) &=& \lim_{u \to 0} \inf_{x \in \Omega}f_u(x),
	\end{eqnarray*}
	which yields the last assertion of the theorem.
\end{proof}

\bibliographystyle{abbrv}
%\bibliography{D:/Submission/BibPureMath1,D:/Submission/BibAppMath1}

\end{document}